\documentclass[a4paper,10pt]{article}
\usepackage[margin=30mm]{geometry}

\usepackage[utf8]{inputenc}
\usepackage[english]{babel}
\usepackage{alphabeta}
\usepackage{lmodern}
\usepackage{amsmath}
\usepackage{amssymb}
\usepackage{amsthm}
\usepackage{graphicx}
\usepackage{fancyhdr}
\usepackage[section]{placeins}
\addto\captionsenglish{}
\usepackage[titletoc]{appendix}
\usepackage{subfig}
\usepackage{wrapfig}
\usepackage{url}
\usepackage{hyperref}
\usepackage{mathrsfs}
\usepackage{mathtools}
\usepackage{cleveref}
\usepackage{algorithm,algorithmic}
\usepackage{dsfont}
\usepackage{wasysym}
\usepackage{mdframed}
\usepackage{lipsum}
\usepackage{enumitem}
\usepackage{schemata}
\usepackage{tikz}
\usepackage{ marvosym }
\usetikzlibrary{tikzmark,calc,decorations.pathreplacing}
\DeclareMathAlphabet\mathbfcal{OMS}{cmsy}{b}{n}
\usepackage{stmaryrd,newlfont}
\usepackage{epigraph}
\usepackage{chngpage}
\usepackage[normalem]{ulem}

\usepackage{color}

\usepackage[symbol]{footmisc}

\newcommand{\R}{\mathds{R}}

\newcommand{\Hb}{\mathcal{H}}

\DeclarePairedDelimiter{\abs}{\lvert}{\rvert}
\DeclarePairedDelimiter{\norme}{\lVert}{\rVert}

\newcommand{\prox}{\text{prox}}
\DeclareMathOperator{\argmin}{arg\,min}

\newcommand{\grandO}[1]{O\mathopen{}\left(#1\right)}
\providecommand{\keywords}[1]{\textbf{\textit{Keywords :}} #1}

\newcommand{\psc}[2]{\langle #1,#2\rangle}

\DeclareMathOperator{\dist}{dist}

\newcommand{\off}[1]{}
\newcommand\red[1]{{\color{black}#1}}

\ifpdf
  \DeclareGraphicsExtensions{.eps,.pdf,.png,.jpg}
\else
  \DeclareGraphicsExtensions{.eps}
\fi

\newtheorem{teo}{Theorem}[section]
\newtheorem{lemma}{Lemma}[section]

\newtheorem{proposition}{Proposition}[section]

\theoremstyle{definition}
\newtheorem{definizione}{Definition}[section]

\newtheorem{esempio}[teo]{Example}
\theoremstyle{remark}
\newtheorem{remark}{Remark}

\newtheorem*{theo*}{Théorème}

\numberwithin{equation}{section}

\title{Convergence rates for the Heavy-Ball continuous dynamics for non-convex optimization, under Polyak-\L ojasiewicz condition.}
\author{Vassilis Apidopoulos\thanks{ MaLGa, DIBRIS, Universit\`a degli Studi di Genova. E-mail: vassilis.apid@gmail.com} \and Nicol\`o Ginatta\thanks{ MaLGa, DIMA, Universit\`a degli Studi di Genova. E-mail: ginatta@dima.unige.it} \and Silvia Villa\thanks{ MaLGa, DIMA, Universit\`a degli Studi di Genova. E-mail: villa@dima.unige.it}
}
\date{}
 
\begin{document}
\maketitle

\begin{abstract} We study  convergence  of the trajectories of the  Heavy Ball dynamical system, with constant damping coefficient, in the framework of convex and non-convex smooth optimization. By using the Polyak-\L ojasiewicz condition, we derive  new linear convergence rates for the associated trajectory, in terms of objective function values, without assuming uniqueness of the minimizer. 
\end{abstract}

\section{Introduction}

Let \(\Hb\) be a Hilbert space and \(F\colon\Hb\to\bar{\R}=\R\) be a  differentiable function with $L$-Lipschitz continuous gradient, such that the set of minimizers is nonempty.
We are interested in the  convergence properties for $t$ going to infinity of the solution-trajectory of the following second-order dynamical system
\begin{equation} \label{DSGD}
	\begin{cases}
		\ddot{x}(t) + \alpha\dot{x}(t) + \nabla F(x(t)) = 0 \\
		x(0) = x_0,\ \dot{x}(0)= 0
	\end{cases}
\end{equation}
where $\alpha>0$. System \eqref{DSGD} is  called {\em heavy ball} system since it  physically describes the motion of a material point (ball) rolling over the graph of the function $F$ and subject to a friction proportional  to the velocity (see \cite{attouch2000heavy,alvarez2000minimizing,attouch2002dynamics}). The constant of proportionality  $\alpha>0$ is called  
\emph{damping parameter} \cite{polyak1964some,alvarez2002second,attouch2000heavy}. In general, systems like \eqref{DSGD} play an important role  in 
various fields such as mechanics and physics, and in optimization (see for example \cite{cabot2009long,cabot2009second,su2016differential,attouch2018fast,attouch2017asymptotic,begout2015damped} and references therein). 

Indeed the dynamical system \eqref{DSGD} has been widely studied in the optimization literature, due to the fact that under suitable assumptions on the function $F$, the solution-trajectory satisfies the minimization property:
\[F(x(t))-\min F \xrightarrow{t\to+\infty}
0.\]
This approach has been especially fruitful to study convergence properties of discrete accelerated first order algorithms \cite{attouch2018fast,apidopoulos2017convergence,apidopoulos2021convergence}. 
As it is shown in \cite{polyak1964some}, for strongly convex functions, system \eqref{DSGD} leads to some linear rates of convergence for $F(x(t))-\min F$, which can be faster, depending on the strong convexity parameter, than those associated to the trajectory of the first order \emph{gradient flow} system \(\dot{x}(t) +\nabla F(x(t))=0\) (see \cite{polyak1963gradient}).
The purpose of this paper is to establish  new convergence properties for the solutions of system \eqref{DSGD}. To the best of our knowledge, there are no results concerning the explicit convergence rates of the trajectories generated by the Heavy ball system \eqref{DSGD}, for objective functions satisfying the Polyak-\L ojasiewicz condition without convexity or uniqueness of the minimizer. In this paper we are addressing this issue and establish worst-case linear convergence rates for the objective function $F(x(t))-\min F$. \red{In addition we extend some of these rates in the case of a convex, but not necessarily differentiable. function $F$. } Polyak-Lojasiewicz condition is a relaxation of strong convexity, and has proven to be especially useful in obtaining linear convergence rates for several dynamical systems usually employed in optimization (see for example~\cite{bolte2007lojasiewicz,bolte2017error,Karimi2016,polyak1963gradient}), as also for training deep neural networks (see for example \cite{pmlr-v97-allen-zhu19a} and \cite{NEURIPS2019_fcdf25d6}).

Systems as \eqref{DSGD} are closely linked with optimization algorithms, such as the heavy-ball algorithm, which are frequently used in various settings in optimization, machine learning and inverse problems. In many cases, the convergence properties of the continuous systems are inherited by their associated numerical schemes, i.e. algorithms. The last observation makes the study of continuous-in-time systems as \eqref{DSGD} very popular, since they can provide powerful insights and tools for the analysis of inertial first order methods \cite{attouch2018fast,attouch2019rate,su2016differential,apidopoulos2017convergence,apidopoulos2021convergence}.

\subsection{Previous work}

The seminal paper by Polyak ~\cite{polyak1964some} shows that,  if $F$ is $\mu$ stronly convex and twice differentiable, the generated trajectory 
converges linearly to the unique minimizer of $F$ as $e^{-2\sqrt{\mu}t}$.
For $\mu<1$, this convergence rate is faster than the one obtained of the gradient flow, that is $e^{-2\mu t}$, see e.g. \cite{polyak1963gradient,polyak1964some,polyak2017lyapunov}. Further studies of system \eqref{DSGD} were also made in \cite{attouch2000heavy}, \cite{alvarez2000minimizing}, \cite{attouch2002dynamics} and \cite{haraux1998convergence}, where convergence properties of the trajectory of solutions of \eqref{DSGD} were established in more general settings: $F$ convex \cite{alvarez2000minimizing}, non-convex \cite{attouch2000heavy} or non-differentiable \cite{attouch2002dynamics}. 

Another interesting fact which was motivated and pointed out in \cite{attouch2000heavy} is that, in contrast with the Gradient flow, the Heavy-Ball \eqref{DSGD}, is a second order (in time) system which is not a descent scheme. This allows the trajectory generated by \eqref{DSGD} to escape possible local minima of non-convex functions, by tuning properly the initial velocity $\dot{x}(0)$ (see for example the corresponding discussion in \cite{attouch2000heavy}). In general, systems like \eqref{DSGD} with other choices of damping parameter became very popular and constitute an active area of research thanks to their nice convergence properties (among the rich literature, one can consult \cite{cabot2009second,su2016differential,attouch2018fast,attouch2019rate,attouch2020first,aujol2018optimal,apidopoulos2018differential} and their possible references). 

Recently, the research directions for system \eqref{DSGD} have focused on studying the case where the objective function \(F\) satisfies weaker geometrical assumptions in order to tackle the minimization problem of a wider family of functions. In this context in \cite{aujol2020HBconvergence,siegel2019accelerated}, it was discovered that one can relax the strong convexity property and the Lipschitz character of the gradient of $F$ and still obtain some linear convergence rates for the associated trajectory of \eqref{DSGD}. For quasi-strongly convex functions (see relation \eqref{quasistrconvex} in Section \ref{sectionpreliminaries}), admitting a unique minimizer, it was shown that the quantity $F(x(t))-{\min}F$ converges as $e^{-\sqrt{2\mu}t}$, where $\mu$ is the quasi-strong convexity parameter and $\alpha=3\sqrt{{\mu}/{2}}$, see Theorem $1$ in \cite{aujol2020HBconvergence}. This result was recently extended in \cite{aujolHBLoja2020}, where the authors derive some linear convergence rates for convex functions admitting a unique minimizer and satisfying the Quadratic growth condition (see \eqref{GC} in Section~ \ref{sectionpreliminaries}), which is equivalent to the Polyak-\L ojasiewicz condition. 

As for the non-convex setting (i.e. the minimizing function $F$ is not necessarily convex), less is known about the convergence properties of system \eqref{DSGD}. In particular, in \cite{begout2015damped} (see also \cite{attouch2020first,attouch2020closed} and \cite{polyak2017lyapunov}), the authors show that for continuously differentiable functions satisfying the Kurdyka-\L ojasiewicz condition (which is a generalization of the Polyak-\L ojasiewicz condition), the solution of system \eqref{DSGD} converges linearly to a minimizer of the corresponding function. However their analysis does not provide explicit formulas for the exponents of these linear rates, which is the main subject of the current paper.  

\subsection{Organization}
The paper is organized as follows: In Section \ref{sectionpreliminaries} we recall some basic definitions and tools concerning the main geometrical assumption on the minimizing function $F$. In Section \ref{sectionresults} we present the main results of the paper concerning the convergence rates of $F(x(t))-\min F$.
In Theorem \ref{coroptimalnon}, we treat the non convex case, while some slight improved rates are given in Theorem \ref{coroptimal} where the minimizing function is additionally supposed to be convex. \red{In Section \ref{sectionnondiff}, we extend some of the results in the convex non-smooth setting. Some numerical experiments are reported in Section \ref{sectionnumerics}.} Section \ref{sectionConv} contains the related convergence analysis and the proofs of the two main theorems. Finally, Appendix \ref{appendixa} contains some basic auxiliary results.

\section{Preliminaries}\label{sectionpreliminaries}

In this section we present some basic definitions and results which will be used in the rest of the paper. 

Let $\Hb$ be a Hilbert space endowed with the scalar product $\psc{\cdot}{\cdot}$ and the associated norm $\norme{\cdot}$. We consider the following minimization problem: 
\begin{equation}\label{minproblem}
    F_{\ast}=\min\{ F(x) ~ : ~ x\in \Hb \}
\end{equation}
where $\Hb$ is a Hilbert space and $F:\Hb\longrightarrow\R$ is a function that satisfies the following assumptions:
\begin{enumerate}
    \item[$\mathcal{A}.1$]\label{A1} $F$ is differentiable with $L$-Lipschitz gradient
    \item[$\mathcal{A}.2$]\label{A2} The set of minimizers $X_{\ast}=\argmin F$ is not empty.
\end{enumerate}

More precisely we are interested in the convergence properties of a solution-trajectory $\{x(t)\}_{t\geq0}$ of the dynamical system \eqref{DSGD}, to a minimizing solution of \eqref{minproblem} 

Under conditions $\mathcal{A}.1$ and $\mathcal{A}.2$, the existence and uniqueness of a strong global solution $x(t)\in \mathcal{C}^{2}(\Hb)$ of the initial value problem \eqref{DSGD} can be guaranteed by the Cauchy-Lipschitz theorem, see \cite[Theorem $11$]{attouch2000heavy}.

Throughout the paper, We will assume that the function $F$ satisfies the \textit{Polyak-Łojasiewicz condition}.
\begin{definizione}\label{2KLdefinition}
Let $F\colon\Hb\to\R$ be a differentiable function with $\argmin F \neq \emptyset$. 
We say that the function $F$ satisfies the Polyak-Łojasiewicz condition, if there exists some constant $\mu>0$, such that the following inequality holds: 
\begin{equation}\label{KL}
   (\forall x\in \Hb) \qquad F(x)-F_* \leq \frac{1}{2\mu}\norme{\nabla F(x)}^{2} \tag{$\mathcal{PL}$}
\end{equation}
\end{definizione}

Condition \eqref{KL} was introduced in the early works in \cite{lojasiewicz1963propriete} (see also \cite{polyak1963gradient}) and can be identified as a particular case of the \L ojasiewicz (or Kurdyka-\L ojasiewicz)  gradient inequality (see for example \cite{bolte2007lojasiewicz,bolte2010characterizations,bolte2017error} and \cite{lojasiewicz1963propriete,Kurdyka1998}). The main difference between our definition and the classical one (see  \cite{Kurdyka1998,bolte2007lojasiewicz,bolte2010characterizations}), is the fact that inequality \eqref{KL} is usually required to hold locally, namely for all points in a neighborhood of a given critical point and in a suitable sublevel set, see \cite{bolte2007lojasiewicz,bolte2010characterizations,bolte2017error} for a throughout analysis and extensions. In this global form, this property has been introduced in \cite{polyak1963gradient}  and is also known  under the name \textit{Polyak-\L ojasiewicz} inequality, see e.g. \cite{polyak1963gradient}, or \cite{Karimi2016}. The global requirement, on the one hand restricts the class of considered functions, but on the other hand allows to obtain global convergence results with explicit constants. Indeed, either for classical dynamical systems or algorithms, condition \eqref{KL}  is intimately connected with the linear convergence rates of the objective function values both in the convex and non-convex setting, see e.g. \cite{polyak1963gradient}, \cite{polyak1964some}, \cite{Karimi2016}, \cite{bolte2017error}, \cite{drusvyatskiy2018error}, \cite{garrigos2017convergence}, \cite{polyak2017lyapunov}, \cite{zhang2020new} and their associated references. In what follows, we will make some remarks about functions satisfying property~\eqref{KL}.
 
The Polyak-\L ojasiewicz condition \eqref{KL} is a relaxation of the strong convexity property of a function, i.e.
\begin{equation}
\label{eq:strconv}
 \exists \mu>0 ~ : ~ \forall (x,y)\in \Hb^2 ~ : \quad   F(y)\geq F(x)+\psc{\nabla F(x)}{y-x} +\frac{\mu}{2}\norme{x-y}^2
\end{equation}

Indeed, from the definition of strong convexity, by considering the minimum with respect to $y\in\Hb$ on both sides, one deduces the the\eqref{KL} condition.

\begin{remark}  In this remark we collect some basic observations about the Polyak-\L ojasiewicz condition in Definition~\ref{2KLdefinition}:
\begin{itemize}
\item Condition \eqref{KL} implies  that every critical point of $F$ is a global minimizer.  
\item Differently from the notion of strong convexity, condition \eqref{KL} does not imply uniqueness of the minimizer, even in the convex case. For instance, in $\R$, consider the function $F(x)=|\abs{x}-1|_{+}^{2}=\bigl(\max\{\abs{x}-1,0\}\bigr)^2$.
\item Condition \eqref{KL}  does not imply that the function is convex, see  Example~\ref{ex:klnonconv} and Figure~\ref{figKLnonconvex}.
\end{itemize}
\end{remark}

\begin{esempio}\label{ex:klnonconv}Let  $f\colon\R^d\to\R$ be a differentiable function and consider $F\colon\R^{d+1}\to\R$ by setting
\begin{equation}\label{equationSilvia}
(\forall (x,y)\in\R^{d+1}) \quad   F(x,y)=(y-f(x))^2.
\end{equation}
Then the set of minimizers of $F$ is the graph $f$, and this set is not convex unless $f$ is affine, thus $F$ is not necessarily convex unless $f$ is affine. On the other hand, condition \eqref{KL} is satisfied with $\mu=2$. Examples of functions in the class described by relation \eqref{equationSilvia} are relevant for deep learning applications, see e.g. \cite{LiuZhuBel20}.

\begin{figure}[ht]
\centering
\includegraphics[scale=0.10,trim={5cm 6.9cm 5cm 5.5cm}]{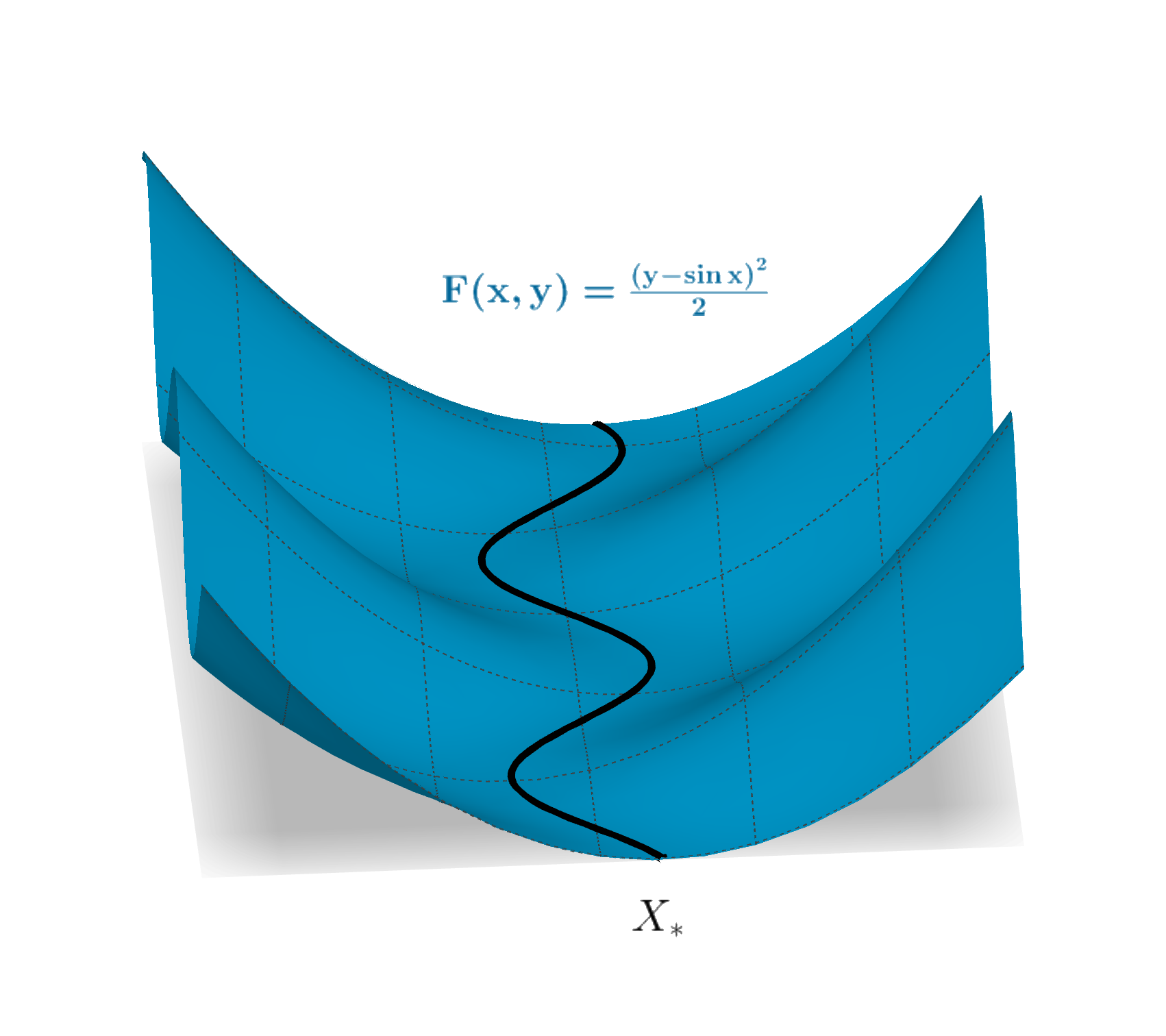}
\caption{Graph of the function $F(x,y)=\frac{(y-\sin x)^2}{2}$ with a non-convex set of minimizers $X_{\ast}=\{(x,y)\in \R^2 ~ : ~ y=\sin{x} \}$.}\label{figKLnonconvex}
\end{figure}
\end{esempio}

Condition \eqref{KL} is also closely connected with other geometrical notions, frequently used in the optimization literature to establish linear convergence rates for approximation methods, such as the \textit{error-bound condition} or \textit{metric subregularity} of the gradient at the origin, i.e.:
\begin{equation}\label{SR}\tag{$\mathcal{EB}$}
\exists \eta >0 \quad : \qquad \eta\text{dist}(x,X_{\ast}) \leq \norme{\nabla F(x)}  \qquad \forall x\in \Hb
\end{equation}
 or the \textit{quadratic growth condition}, i.e.: 
\begin{equation}\label{GC}\tag{$\mathcal{QG}$}
  \exists \theta >0 \quad : \qquad  \frac{\theta}{2}\text{dist}(x,X_{\ast})^2 \leq F(x)-F_{\ast} \qquad \forall x\in \Hb
\end{equation}
where dist$(x,C)=\inf\{\norme{x-y} ~ : ~ y\in C\}$ denotes the distance of a point $x\in \Hb $ from a set $C\subset \Hb$.

Another notion which has also been used to establish linear convergence of first order methods is the notion of \textit{quasi-strong convexity}, 
which was introduced in \cite{necoara2019linear}  as a relaxation of the strong convexity property of a function. It has also been 
used to study the asymptotic properties of the trajectories of system~\eqref{DSGD} in the recent work \cite{aujol2020HBconvergence}. 

A differentiable function $F\colon\Hb\to\R$ is called $\beta$-quasi-strongly convex, if there exists some constant $\mu>0$ such that for all $x\in \Hb$ and for all projections $\bar{x}$ of $x$ on $X_*$, it holds:
\begin{equation}\label{quasistrconvex}\tag{q$\mathcal{SC}$}
    F(\bar{x})\geq F(x)+\psc{\nabla F(x)}{\bar{x}-x} +\frac{\beta}{2}\norme{x-\bar{x}}^2.
\end{equation}
Notice however that a projection $\bar{x}$ of $x$ on $X_*$ may not exist when $\Hb$ is infinitely dimensional, e.g. if  $X_{\ast}$ is not a Chebyshev set. Unless there are no additional assumptions such as convexity or weakly closeness of $X_{\ast}$, or finite dimensionality of the space $\Hb$, relation \eqref{quasistrconvex} may have an empty meaning. In what follows, whenever employing relation \eqref{quasistrconvex}, we will refer to functions $F$  such that the projection of a point onto $X_{\ast}$ is nonempty.  Note however that, given $x\in \Hb$, there may be more than one projection of $x$ on $X_{\ast}$ .  Unlike strong convexity of a function, it is worth mentioning that the quasi-strong convexity does not imply neither convexity nor uniqueness of a minimizer. 

Below we present some known results about the interplay between conditions \eqref{KL}, \eqref{SR}, \eqref{GC} and \eqref{quasistrconvex} in different settings. In the convex setting, conditions \eqref{KL}, \eqref{SR} and \eqref{GC} are all equivalent, while  quasi-strong convexity is stronger than the Polyak-\L ojasiewicz condition \eqref{KL}. In particular, it can be shown that the class of quasi-strongly convex functions is a subclass of functions satisfying condition \eqref{KL}. 

\begin{proposition}\label{prop: geom_interplay}
Let $F\colon\Hb\to\R$ be a continuously differentiable function with $\argmin F \neq \emptyset$. Then the following implications hold true:
\begin{equation}\label{1stimplication}
 \eqref{quasistrconvex} \implies   \eqref{KL} \red{\implies \eqref{GC}} 
\end{equation}
with $\theta=\mu=\beta$.
\begin{enumerate}
\item If $F$ has $L$-Lipschitz gradient, then:
\begin{equation}\label{2ndimplication}
\eqref{SR} \implies  \eqref{KL}
\end{equation}
with $\mu=\frac{\eta^{2}}{L}$.
\item If $F$ is convex, then:  
\begin{equation}\label{3rdimplication}
 \eqref{GC} \implies  \eqref{SR} \implies  \eqref{KL}
\end{equation}
with $\eta=\frac{\theta}{2}$ and $\mu=\frac{\theta}{4}$.
\item If $F$ is convex with $L$-Lipschitz gradient, then:
 \begin{equation}\label{4thimplication}
     \eqref{KL} \implies  \eqref{quasistrconvex}
 \end{equation}
 with $\beta=\frac{\mu^2}{L}$
\end{enumerate}
\end{proposition}

Proposition \ref{prop: geom_interplay} collects some already known  results. 
\begin{proof}
For the proof of the first implication in \eqref{1stimplication}, we observe that by using the Cauchy-Schwarz inequality for the scalar product in the quasi-strong convexity condition \eqref{quasistrconvex} and then the Young's inequality for the product, for any  $x\in \Hb$ and a projection $\bar{x}$ of $x$ onto $X_*$ and any $\varepsilon>0$, we have:
\begin{equation}
\begin{aligned}
       F(x)-F(\bar{x}) &\leq \psc{\nabla F(x)}{x-\bar{x}} -\frac{\beta}{2}\norme{x-\bar{x}}^2 \\ &\leq \norme{\nabla F(x)}\norme{x-\bar{x}} -\frac{\beta}{2}\norme{x-\bar{x}}^2 \\
       &\leq  \frac{1}{2\varepsilon}\norme{\nabla F(x)}^{2} + \frac{(\varepsilon-\beta)}{2}\norme{x-\bar{x}}^2
       \end{aligned}
\end{equation}
which, by choosing $\varepsilon=\beta$, is the \eqref{KL} condition.

The second implication of \eqref{1stimplication} can be found on \cite[Theorem $27$]{bolte2017error} and \cite[Theorem $1$]{zhang2020new}. 

The implication \eqref{2ndimplication} can be found in \cite[Theorem $2$]{Karimi2016}, and the ones in \eqref{3rdimplication}, in \cite[Theorem $1$]{zhang2020new} or in \cite[Proposition $3.3$]{garrigos2017convergence} (see also \cite{bolte2010characterizations,bolte2017error,Karimi2016}).

Finally the last implication in \eqref{4thimplication} can be found in \cite[Lemma $2$]{aujol2020HBconvergence} and we give here a brief proof for sake of completeness.

 Indeed, since $F$ is convex with $L$-Lipschitz gradient satisfying condition \eqref{KL}, then by the Baillon-Haddad Theorem (see \eqref{convexLipschitz2} in Lemma \ref{convexLipschitz} in Appendix \ref{appendixa}) we have:
\begin{equation}\label{equationpropositiongeom22}
F(x)-F_{\ast}\leq \psc{\nabla F(x)}{x-\bar{x}}-\frac{1}{2L}\norme{\nabla F(x)}^{2}\leq \psc{\nabla F(x)}{x-\bar{x}}-\frac{\mu}{L}\big(F(x)-F_{\ast}\big)
\end{equation}
Since \eqref{KL} implies \eqref{GC} with $\theta$=$\mu$, by using \eqref{GC} in \eqref{equationpropositiongeom22}, we find:
\begin{equation}
F(x)-F_{\ast}\leq \psc{\nabla F(x)}{x-\bar{x}}-\frac{\mu^2}{2L}\norme{x-\bar{x}}^{2}.
\end{equation}
which shows that $F$ is $\frac{\mu^2}{L}$-quasi strongly convex.

\end{proof}

\begin{remark}\label{remarkquasistrconvexty} 
Notice that for non-convex functions the inverse implication of \eqref{4thimplication} in Proposition \eqref{prop: geom_interplay} does not hold true in general. Motivated by an example given in \cite{Karimi2016}, for any $\varepsilon>0$, one can consider the following non-convex function:
\begin{equation}\label{nonqstrconvex}
F(x)=\begin{cases}
(x+\varepsilon)^2+3\sin^2(x+\varepsilon) &~ \text{ if } ~ x<-\varepsilon \\ 0 &~ \text{ if } ~ -\varepsilon\leq x\leq \varepsilon \\ (x-\varepsilon)^2+3\sin^2(x-\varepsilon) &~ \text{ if } ~ x>\varepsilon
\end{cases}
\end{equation}
The function $F$ is continuously differentiable with $L$-Lipschitz gradient with $L\leq 14$, satisfies \eqref{KL} with some constant $\mu=\frac{1}{32}$. However this function is not (globally) quasi-strongly convex (see also Figure \ref{figure1nonconvex}).
\end{remark}
\begin{figure}[ht]
\centering
\includegraphics[scale=1.15,trim={0cm 0cm 0cm 4mm}]{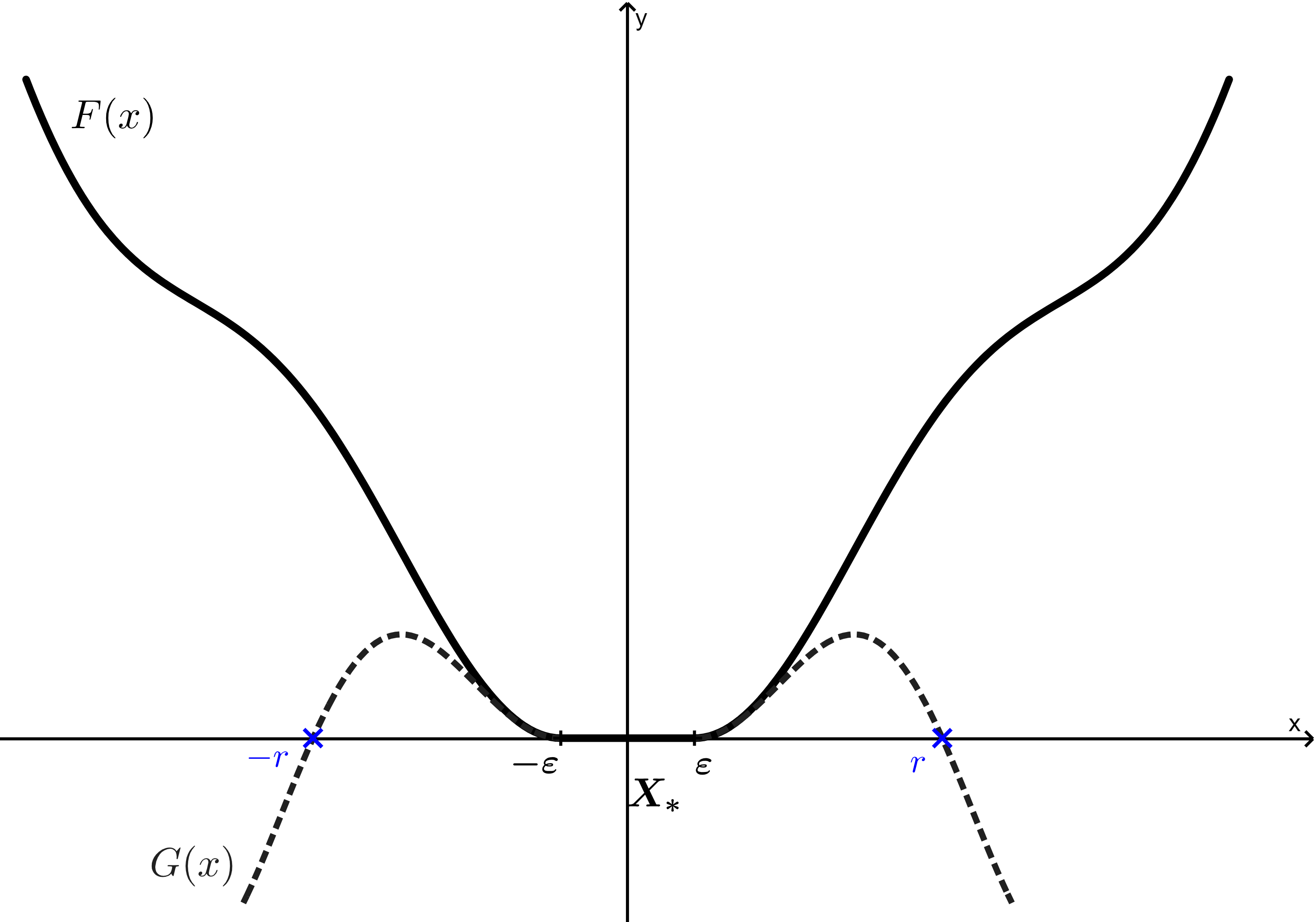}
\caption{
In solid line, the graph of the function $F(x)$ as defined in \eqref{nonqstrconvex} for some $\varepsilon>0$. In dashed line the function $G(x)=\langle F^{\prime}(x),x-\bar{x}\rangle - F(x)$. One can observe that there exists some neighborhood $(-r,r)$, outside of which $G(x)\leq 0$. This implies that it cannot exist any $\mu>0$ such that $G(x)-\frac{\mu}{2}|x-\bar{x}|^2\geq 0$ outside of $(-r,r)$, which implies that the function is not quasi-strongly convex.}
\label{figure1nonconvex}
\end{figure}

\section{Main results}\label{sectionresults}
In this section we present the main results about convergence of the trajectories of the dynamical system \eqref{DSGD}, 
both in the nonconvex and convex case. In Theorem~\ref{coroptimalnon} and Theorem~\ref{coroptimal} we provide some upper bounds for the decay of the 
trajectory of \eqref{DSGD}, in terms of objective function values, for a specific choice of the damping parameter $\alpha$. 
This choice of $\alpha$ is "optimal" in the sense that it is the one that ensures the fastest decay rate of the trajectory, according to the upper bounds found in the convergence analysis (see Theorems \ref{basicteorates} and \ref{basicteoratesconvex} in Section \ref{sectionConv}).

\begin{teo}\label{coroptimalnon}
Let $F\colon \Hb\to\R$ be a differentiable function with $L$-Lipschitz gradient. Assume that $\argmin F\neq \emptyset$ and $F$ also satisfies \eqref{KL} with $\mu>0$ and denote $F_{\ast}=\min F$ and $\kappa=\frac{L}{\mu}$. If $\big(x(t)\big)_{t\geq0}$ is the solution-trajectory of the dynamical system \eqref{DSGD}, then for all $\varepsilon>0$, the following bounds hold true:
\begin{itemize}
    \item If \(\kappa\leq\frac98\) and \(\alpha =\frac{\sqrt{\mu}}{2\sqrt{2}}\left(5+\sqrt{9-8\kappa}\right)-\frac\varepsilon2\), then 
    \begin{equation}\label{ratescormu<L2}
        F(x(t))-F_\ast  \leq \left(F(x(0))-F_\ast\right)\left(\frac{4\varepsilon+2\sqrt{2\mu}}{\varepsilon}\right)e^{-\left(\sqrt{2\mu}-\varepsilon\right)t}
    \end{equation}
    \item If \(\kappa>\frac98\) and $\alpha=(2\sqrt{\kappa}-\sqrt{\kappa-1})\sqrt{\mu}$, then
    \begin{equation}\label{ratescormuequalL}
        F(x(t))-F_\ast \leq \left(F(x(0))-F_\ast\right)\frac{4\sqrt{\kappa-1}\left(3\sqrt{\kappa-1}+\sqrt{\kappa}\right)}{8\kappa-9}e^{-2\left(\sqrt{\kappa}-\sqrt{\kappa-1}\right)\sqrt{\mu}t}
    \end{equation} 
\end{itemize}
\end{teo}

The next Theorem shows that for convex functions, one can obtain slightly better results. 

\begin{teo}\label{coroptimal}
Let $F\colon \Hb\to\R$ be a convex differentiable function with $L$-Lipschitz gradient. Assume that $\argmin F\neq \emptyset$ and $F$ also satisfies \eqref{KL} with $\mu>0$ and denote $F_{\ast}=\min F$ and $\kappa=\frac{L}{\mu}$. If $\big(x(t)\big)_{t\geq0}$ is the solution-trajectory of the dynamical system \eqref{DSGD}, then for all $\varepsilon>0$, the following bounds hold true:
\begin{itemize}
 \item If $\kappa=1$,  $\varepsilon>0$ and $\alpha=2\sqrt{\mu}-\varepsilon$, then
 \red{
 \begin{equation}
 \norme{\nabla F(x(t))}^{2} \leq 2\big(F(x(0))-F_{\ast}\big)\bigg(1+\frac{\sqrt{\mu}}{\varepsilon}\bigg)e^{-2\bigl(\sqrt{\mu}-\varepsilon\bigr)t}
 \end{equation}
and
}
\begin{equation}\label{ratescormuequalL2}
F(x(t))-F_{\ast}  \leq\frac{\big(F(x(0))-F_{\ast}\big)}{\mu}\bigg(1+\frac{\sqrt{\mu}}{\varepsilon}\bigg)e^{-2\bigl(\sqrt{\mu}-\varepsilon\bigr)t}
\end{equation} 
    \item If $\kappa>1$, and $\alpha=(2\sqrt{\kappa}-\sqrt{\kappa-1})\sqrt{\mu}$, then
\red{    
\begin{equation}\label{nondiffratescormu<L1}
\norme{\nabla F(x(t))}^{2} \leq 2\big(F(x(0))-F_{\ast}\big)\bigg(1+\sqrt{\frac{\kappa}{\kappa-1}}\bigg)e^{-2\bigl(\sqrt{\kappa}-\sqrt{\kappa-1}\bigr)\sqrt{\mu}t}
\end{equation}
     and
     }
\begin{equation}\label{ratescormu<L}
F(x(t))-F_{\ast} \leq\frac{\big(F(x(0))-F_{\ast}\big)}{\mu}\bigg(1+\sqrt{\frac{\kappa}{\kappa-1}}\bigg)e^{-2\bigl(\sqrt{\kappa}-\sqrt{\kappa-1}\bigr)\sqrt{\mu}t}
\end{equation}
\end{itemize}
\end{teo}

Before presenting the convergence analysis and the corresponding proofs, let us make some comments related to Theorems \ref{coroptimalnon} and \ref{coroptimal}.

\begin{remark}[Localization on the sub-level sets]\label{remlocalization}
While Theorem \ref{coroptimalnon} refers to functions satisfying the Polyak-\L ojasiewicz condition globally (see Definition \ref{2KLdefinition}), the results of the aforementioned theorem still hold true for any function $F$ that satisfies \eqref{KL} on the sublevel set \(\Omega=\{ x\in \Hb ~ : ~ F(x) \leq F(x(0)) \}\), which is invariant with respect to the system \eqref{DSGD} (i.e. $x(t)\in \Omega$, $\forall t\geq 0$).

Indeed by considering the total energy of system \eqref{DSGD}, $U(t)=F(x(t))+\frac{1}{2}\norme{\dot{x}(t)}^2$, we find:
\begin{equation}
    \dot{U}(t)=\psc{\nabla F(x(t))+\ddot{x}(t)}{\dot{x}(t)}\overset{\eqref{DSGD}}{=}-\alpha\norme{\dot{x}(t)}^2
\end{equation}
which shows that $U(t)$ is non-increasing. By using the non-increasing property of $U$ we deduce that
\(F(x(t))\leq F(x(0))\), for all $t\geq0$ and therefore the trajectory \(\{x(t)\}_{t\geq 0}\) generated by \eqref{DSGD} remains in the sublevel set $\Omega$, where $F$ satisfies \eqref{KL} and the decay rates found in Theorem \ref{coroptimalnon} remain valid. 
\end{remark}

\begin{remark}
In both Theorems \ref{coroptimalnon} and \ref{coroptimal}, one can observe that, the associated worst-case linear convergence rates are affected by the magnitude of the condition number $\kappa=\frac{L}{\mu}\geq 1$ (equivalently the magnitude of the difference $L-\mu\geq0$). In particular, while in the case $\kappa=1$ (i.e. $\mu=L$) the rate in \eqref{ratescormuequalL} is worst-case optimal (see next Remark \ref{remarkoptimalquadratic}), in the case $\kappa>1$, the exponents in \eqref{ratescormuequalL} and \eqref{ratescormu<L}  can become small, if $\kappa=\frac{L}{\mu}>>1$.
\end{remark}

\begin{remark}[Optimality of rates in the case $\mu=L$ for convex functions]\label{remarkoptimalquadratic} As remarked in \cite{ghisi2016remarkable}, if $\Hb=\R$ and $F(x)=\frac{\mu}{2}\abs{x}^2$ (quadratic function in one-dimensional case), the system \eqref{DSGD} reduces to a linear ODE with constant coefficients whose solution $x(t)$ has an explicit form. In that case it can be shown that  for all $\varepsilon>0$:
\begin{equation}
    F(x(t))-F_{\ast} \leq C_{\varepsilon}t^{-(r(\alpha)-\varepsilon)t} \qquad \text{with : } ~ r(\alpha)=\frac{\alpha-\min\{0,\sqrt{\alpha^2-4\mu}\}}{2}
\end{equation}
and $\sup\{r(\alpha) ~ : ~ \alpha>0 )\}=2\sqrt{\mu}$. This shows that the estimate \eqref{ratescormuequalL2}, in the case $\mu=L$ of Theorem \ref{coroptimal}, recovers the -worst case- optimal decay rate in the class of convex functions satisfying condition \eqref{KL}.
\end{remark}

\begin{remark}[Comparison with Gradient flow and the works \cite{aujolHBLoja2020,aujol2020HBconvergence}]

\red{
The first-order in time Gradient flow system, i.e. \(\dot{x}(t) +\nabla F(x(t))=0\), provides linear convergence rates for the objective function values with a corresponding convergence factor equal to $2\mu$ (see e.g.
\cite{polyak1963gradient,polyak1964some,polyak2017lyapunov}) both in convex and non-convex setting.
By comparing the convergence factor $2(\sqrt{\kappa}-\sqrt{\kappa-1})\sqrt{\mu}$ found in \eqref{ratescormuequalL} and \eqref{ratescormu<L}  of Theorems \ref{coroptimalnon} and \ref{coroptimal} (respectively), we have: 
\begin{equation}\label{eq:211210a}
\begin{aligned}
    \max\big\{2(\sqrt{\kappa}-\sqrt{\kappa-1})\sqrt{\mu},2\mu\big\} =  \\ =&\begin{cases} 2(\sqrt{\kappa}-\sqrt{\kappa-1})\sqrt{\mu} & ~ \text{ if } ~  2\sqrt{L}-1 \leq \mu \leq 1 \\ 2\mu & ~ \text{ if } ~ (\mu<2\sqrt{L}-1) \text{ or } (\mu>1)
    \end{cases}
\end{aligned}
\end{equation}
describing the regimes for parameters $\mu$ and $L$, for which the worst-case decay rate of each method is faster than the other.

It is also worth mentioning that the Lipschitz character of the gradient is not necessary to deduce the linear rates for the Gradient flow, in contrast to our setting. Nevertheless, Theorem \ref{coroptimalnon} is the first result that states explicit rates in the non-convex setting for the Heavy ball dynamical system. 

}

Concerning the convex setting, similar results are obtained in \cite{aujolHBLoja2020}, where the authors study the system \eqref{DSGD} for convex differentiable functions, with a unique minimizer, satisfying the Quadratic growth condition \eqref{GC} with $\theta>0$ (recall here that in the convex setting the Quadratic growth condition \eqref{GC} is equivalent to the Polyak-\L ojasiewicz condition \eqref{KL} with the same constant ($\alpha=\mu$), see \eqref{1stimplication} in Proposition \ref{prop: geom_interplay} of Section \ref{sectionpreliminaries}). In particular, for the aforementioned class of functions, they provide linear decay rates for the objective error $F(x(t))-\min F$, with a worst-case optimal convergence factor of $(2-\sqrt{2})\sqrt{\mu}$, achieved for $\alpha=(2-\frac{\sqrt{2}}{2})\sqrt{\mu}$ (see Theorem and Corollary $1$ in \cite{aujolHBLoja2020}). Comparing these rates with those found in Theorem \ref{coroptimal}, one has the following:
\begin{equation}\label{eq:210311a}
\begin{aligned}
    \max\big\{2(\sqrt{\kappa}-\sqrt{\kappa-1})\sqrt{\mu},(2&-\sqrt{2})\sqrt{\mu}\big\} =  \\ =&\begin{cases} 2(\sqrt{\kappa}-\sqrt{\kappa-1})\sqrt{\mu} & ~ \text{ if } ~ \kappa \leq \kappa_{\ast}=\frac{19+6\sqrt{2}}{8} \\ (2-\sqrt{2})\sqrt{\mu} & ~ \text{ if } ~ \kappa > \kappa_{\ast}=\frac{19+6\sqrt{2}}{8}
    \end{cases}
\end{aligned}
\end{equation}
From equation~\eqref{eq:210311a} one can see that if $\kappa\leq \kappa_{\ast}$ the rate in Theorem \ref{coroptimal} is faster, while in the case $\kappa>\kappa_{\ast}$, the rate obtained in \cite{aujolHBLoja2020} is better and is proven by exploiting a different Lyapunov function, relying though on the uniqueness of the minimizer of the function. 

In the same spirit, in \cite{aujol2020HBconvergence} the authors prove similar linear convergence rates, in terms of objective function values, for convex and $\beta$-quasi-strongly convex functions admitting a unique minimizer. Theorem $1$ in \cite{aujol2020HBconvergence} states that the best rate for the objective value function is $\exp{(-\sqrt{2\beta}t)}$.  Since condition \eqref{KL} with $\mu>0$ for convex functions with $L$-Lipschitz gradient implies quasi-strong convexity with parameter $\beta=\frac{\mu^2}{L}$ (see \eqref{4thimplication} in Proposition \ref{prop: geom_interplay}), we have $\sqrt{2\beta}=\sqrt{\frac{2\mu}{\kappa}}$. As done before, a straightforward comparison of the two factors $2(\sqrt{\kappa}-\sqrt{\kappa-1})\sqrt{\mu}$ and $\sqrt{\frac{2\mu}{\kappa}}$, shows that:
\begin{equation}
\begin{aligned}
    \max\big\{2(\sqrt{\kappa}-\sqrt{\kappa-1})\sqrt{\mu},&\sqrt{\frac{2\mu}{\kappa}}\big\} =  \\ =&\begin{cases} 2(\sqrt{\kappa}-\sqrt{\kappa-1})\sqrt{\mu} & ~ \text{ if } ~ \frac{L}{\mu} \leq \kappa_{\ast}:=\frac{1+\sqrt{2}}{2} \\ \sqrt{\frac{2\mu}{\kappa}} & ~ \text{ if } ~ \frac{L}{\mu} > \kappa_{\ast}=\frac{1+\sqrt{2}}{2}
    \end{cases}
\end{aligned}
\end{equation}
Theorem $2$ in \cite{aujol2020HBconvergence}) states improved rated when $F$ has Lipschitz gradient. They are given  via the formula $\exp{(-\frac{2}{3}(1+\frac{2}{3}\frac{9\mu^{2}-2\alpha^2}{9L+3\mu-\frac{2}{3}\alpha^2})t)}$, for $\alpha\leq 3\sqrt{\frac{\mu}{2}}$. However, the optimal rate with respect to $\alpha$ in the last formula cannot be explicitly computed. A straightforward comparison with our results is therefore difficult to  establish. 
\end{remark}

\section{Non-smooth setting}\label{sectionnondiff}

In this section we extend some of the previous results to the case of a convex, proper and lower semi-continuous function $F: \Hb \to \overline{\R}=\R\cup\{+\infty\}$, with $X_{\ast}=\argmin F\neq \emptyset$. In this context, we make use of the Moreau envelope  and the proximal operator of $F$ defined (respectively) for any $\lambda>0$, as follows (see e.g. \cite[Definitions $12.20$ and $12.23$]{bauschke2011convex}): 
\begin{align}
(\forall x\in \Hb) \qquad F_{\lambda}(x)&=\underset{y\in \Hb}{\min}\left\{F(y)+\frac{1}{2\lambda}\norme{y-x}^2\right\} \\
\prox_{\lambda F}(x)&=\underset{y\in \Hb}{\argmin}\left\{F(y)+\frac{1}{2\lambda}\norme{y-x}^2\right\}
\end{align}

One of the major advantage of the Moreau envelope $F_{\lambda}$ is that is a convex, continuously differentiable function with $\frac{1}{\lambda}$-Lipschitz gradient (see e.g. \cite[Proposition $12.29$]{bauschke2011convex}). Thus considering he Heavy ball system \eqref{DSGD} for $F_{\lambda}$ is still possible even if $F$ is not smooth or differentiable.

Formally, we state all the useful properties regarding $F_{\lambda}$ and $\prox_{\lambda F}$, that are used in this work in Lemma \ref{lemmaMoreau}. For further details and proofs about the Moreau envelope and the proximal operator we address the interested reader to \cite{rockafellar2009variational} and \cite{bauschke2011convex}.
\begin{lemma}\label{lemmaMoreau}
Let $F:\Hb\to\overline{\R}$ be a convex, proper and lower semi-continuous function with $\argmin F\neq \emptyset$. Let $\lambda>0$ and $F_{\lambda}$ and $\prox_{\lambda F}$ the Moreau envelope and proximal operator of $F$ (respectively). The following assertions hold true for all $x\in \Hb$.
\begin{enumerate}
\item \(\underset{x\in \Hb}{\min} F(x) = \underset{x\in \Hb}{\min}F_{\lambda}(x)\) ~ and ~ \(\underset{x\in \Hb}{\argmin} F(x) = \underset{x\in \Hb}{\argmin}F_{\lambda}(x)\).
\item  \(F_{\lambda}(x)=F(\prox_{\lambda F}(x))+\frac{1}{2\lambda}\norme{x-\prox_{\lambda F}(x)}^2 \leq F(x)\).
\item \(\frac{1}{\lambda}\left(x-\prox_{\lambda F}(x)\right) \in \partial F(x)\).
\item \(F_{\lambda}\) is differentiable with \(\frac{1}{\lambda}\)-Lipschitz gradient and $\nabla F_{\lambda}(x)=\frac{1}{\lambda}\left(x-\prox_{\lambda F}(x)\right)$. 
\end{enumerate}
\end{lemma}

Notice that in this setting the \eqref{KL} condition on $F$ can be generalized as follows (see e.g. \cite[Paragraph $2.3$]{bolte2017error}): 
\begin{equation}\label{nondiffKL}
F(x)-F_{\ast} \leq \frac{1}{2\mu}\dist(0,\partial F(x))^2 \tag{ns-$\mathcal{PL}$}
\end{equation}
where $\partial F$ is the subdifferential of $F$. It is also worth mentioning that condition \eqref{nondiffKL} and the growth condition \eqref{GC} are still equivalent in this setting ($F$ convex) with same relations between the two parameters as the ones described in \eqref{1stimplication} and \eqref{3rdimplication} of Proposition \ref{prop: geom_interplay} (see e.g. \cite[Theorem $27$]{bolte2017error} and \cite[Theorem $1$]{zhang2020new}).

\begin{proposition}\label{propequivMoreau}
Let $\lambda>0$, $F:\Hb \to \overline{\R}$ be a convex, proper and lower semi-continuous function  and $F_{\lambda}$ its Moreau envelope. Then the following assertions hold true:
\begin{enumerate}
\item If $F$ satisfies \eqref{nondiffKL} with parameter $\mu$, then $F_{\lambda}$ satisfies \eqref{KL} with parameter \(\frac{\mu}{\lambda\mu+1}\).
\item If $F_{\lambda}$ satisfies \eqref{KL} with parameter $\mu$, then $F$ satisfies \eqref{nondiffKL} with parameter \(\frac{\mu}{4}\).
\end{enumerate}
\end{proposition}

\begin{proof}
1. By using properties $1.$ and $2.$ in Lemma \ref{lemmaMoreau}), we have:
\begin{equation}\label{equationequivalence}
\begin{aligned}
F_{\lambda}(x)-\min F_{\lambda}&=F(\prox_{\lambda F}(x))+\frac{1}{2\lambda}\norme{x-\prox_{\lambda F}(x)}^2 - F_{\ast} \\
& \leq \frac{1}{2\mu}\dist(0,\partial F(\prox_{\lambda F}(x)))^2+\frac{1}{2\lambda}\norme{x-\prox_{\lambda F}(x)}^2 
\end{aligned}
\end{equation}
where in the last inequality we used the fact that $F$ satisfies condition \eqref{nondiffKL}.

Since $\frac{1}{\lambda}(x-\prox_{\lambda F}(x))\in \partial  F(\prox_{\lambda F}(x))$ by property $3.$ in Lemma \ref{lemmaMoreau}, from \eqref{equationequivalence}, we obtain:
\begin{equation}
\begin{aligned}
F_{\lambda}(x)-\min F_{\lambda} & \leq \frac{1}{2\mu}\dist(0,\partial F(\prox_{\lambda F}(x)))^2+\frac{1}{2\lambda}\norme{x-\prox_{\lambda F}(x)}^2  \\
& \leq \frac{1}{2\mu}\norme{\frac{1}{\lambda}\left(x-\prox_{\lambda F}(x)\right)}^2 +\frac{1}{2\lambda}\norme{x-\prox_{\lambda F}(x)}^2  \\
& = \left(\frac{1}{2\mu}+\frac{\lambda}{2}\right)\norme{ \frac{1}{\lambda}\left(x-\prox_{\lambda F}(x)\right)}^2  \\
& = \frac{(1+\lambda\mu)}{2\mu}\norme{\nabla F_{\lambda}(x)}^2
\end{aligned}
\end{equation}
which allows to conclude that $F_{\lambda}$ satisfies \eqref{KL} with parameter \(\frac{\mu}{\lambda\mu+1}\).

2. From \eqref{1stimplication} of Proposition \ref{prop: geom_interplay} $F_{\lambda}$ satisfies \eqref{GC} with parameter $\mu$, thus, by using properties $1.$ and $2.$ in Lemma \ref{lemmaMoreau} we obtain:
\begin{equation}\label{gqforF}
\begin{aligned}
\frac{\mu}{2}\dist(x,\argmin F_{\lambda})^2 &\leq F_{\lambda}(x)-\min F_{\lambda}=F_{\lambda}(x)-\min F \\
& \leq F(x) -\min F
\end{aligned}
\end{equation}

From \eqref{gqforF}, it follows that $F$ satisfies \eqref{GC}, with parameter $\mu$, therefore $F$ satisfies \eqref{nondiffKL} with parameter $\frac{\mu}{4}$ (see e.g. \cite[Theorem $1$]{zhang2020new}).
\end{proof}

\begin{teo}\label{basicteonondiff}
Let $F\colon \Hb\to\overline{\R}$ be a convex, proper and lower semi-continuous function with $\argmin F\neq \emptyset$ satisfying \eqref{KL} with $\mu>0$. Let also $\lambda>0$ and  $F_{\lambda}$ and $\prox_{\lambda F}$ be the Moreau envelope and the proximal operator of $F$ (respectively) and \(x_{\lambda}(t)\) the trajectory generated by the Heavy ball system \eqref{DSGD} with $F_{\lambda}$, with \(\alpha=\frac{2\lambda\mu+1+\sqrt{\lambda\mu+1}}{\sqrt{\lambda}\left(\lambda\mu+1+\sqrt{\lambda\mu+1}\right)}\). Then the following bounds hold true:
\begin{equation}\label{nondiffratescormu<L}
F(\prox_{\lambda F}(x_{\lambda}(t)))-F_{\ast} \leq \big(F_{\lambda}(x(0))-F_{\ast}\big)\left(\lambda+\frac{1}{\mu}\right)\bigg(1+\sqrt{\lambda\mu+1}\bigg)e^{-\frac{2\mu\sqrt{\lambda}}{\mu\lambda+1+\sqrt{\mu\lambda+1}}t}
\end{equation}
If in addition $F$ is $M$-Lipschitz, then:
\begin{equation}\label{nondiffratescorgradient1}
\begin{aligned}
F(x_{\lambda}(t))-F_{\ast} & \leq 2\max\left\{\sqrt{2}M\lambda,\frac{C(\lambda,\mu)}{\mu}e^{-\frac{\mu\sqrt{\lambda}}{\mu\lambda+1+\sqrt{\mu\lambda+1}}t}\right\}C(\lambda,\mu)e^{-\frac{\mu\sqrt{\lambda}}{\mu\lambda+1+\sqrt{\mu\lambda+1}}t} \\
& =\grandO{e^{-\frac{\mu\sqrt{\lambda}}{\mu\lambda+1+\sqrt{\mu\lambda+1}}t}}
\end{aligned}
\end{equation} 
where  \(C(\lambda,\mu)=\sqrt{\big(F_{\lambda}(x_{\lambda}(0))-F_{\ast}\big)\bigg(1+\sqrt{\lambda\mu+1}\bigg)}\)
\end{teo}

\begin{proof}[\textbf{Proof of Theorem \ref{basicteonondiff}}]
On the one hand, from property $2.$ in Lemma \ref{lemmaMoreau} we have $F(\prox_{\lambda F}(x_{\lambda}(t)))-F_{\ast}\leq F_{\lambda}(x_{\lambda}(t))-\min F_{\lambda}$. On the other hand, from Proposition \ref{propequivMoreau}, it follows that $F_{\lambda}$ satisfies \eqref{KL} with parameter $\frac{\mu}{\lambda\mu+1}$ and since it is convex with $\frac{1}{\lambda}$-Lipschitz gradient, we can directly apply Theorem \ref{coroptimal} and deduce \eqref{nondiffratescormu<L} from \eqref{ratescormu<L}.

If now we assume that $F$ is $M$-Lipschitz, by the characterization of $F_{\lambda}$ and $\prox_{\lambda F}$ (see properties $1$, $2$ and $4$ in Lemma  \ref{lemmaMoreau}), we have:
\begin{equation}\label{proofgradientnondiff}
\begin{aligned}
F(x_{\lambda}(t))-F_{\ast} -\left(F_{\lambda}(x_{\lambda}(t))-F_{\ast}\right)&=F(x_{\lambda}(t))-F(\prox_{\lambda F}(x_{\lambda}(t)))-\frac{1}{2\lambda}\norme{x_{\lambda}(t)-\prox_{\lambda F}(x_{\lambda}(t))}^2 \\
&\leq M\norme{x_{\lambda}(t)-\prox_{\lambda F}(x_{\lambda}(t))}-\frac{1}{2\lambda}\norme{x_{\lambda}(t)-\prox_{\lambda F}(x_{\lambda}(t))}^2 \\
&= M\lambda\norme{\nabla F_{\lambda}(x_{\lambda}(t))} - \frac{\lambda}{2}\norme{\nabla F_{\lambda}(x_{\lambda}(t))}^2 \\
& \leq M\lambda\norme{\nabla F_{\lambda}(x_{\lambda}(t))} - \frac{\lambda\mu}{\lambda\mu+1}\left(F_{\lambda}(x_{\lambda}(t))-F_{\ast}\right)
\end{aligned}
\end{equation}
where in the second inequality we used the $M$-Lipschitz character of $F$ and in the last one that $F_{\lambda}$ satisfies \eqref{nondiffKL} with parameter \(\frac{\mu}{\lambda\mu+1}\).

Thus from \eqref{proofgradientnondiff}, it follows that
\begin{equation}
F(x_{\lambda}(t))-F_{\ast} \leq M\lambda\norme{\nabla F_{\lambda}(x_{\lambda}(t))} + \frac{1}{\lambda\mu+1}\left(F_{\lambda}(x_{\lambda}(t))-F_{\ast}\right).
\end{equation}
Inequality \eqref{nondiffratescorgradient1}, follows from the bounds \eqref{nondiffratescormu<L1} and \eqref{ratescormu<L} of Theorem \ref{coroptimal} for \(\norme{\nabla F_{\lambda}(x_{\lambda}(t))}\) and \(F_{\lambda}(x_{\lambda}(t))-F_{\ast}\) respectively.
\end{proof}

\section{Numerical experiments}\label{sectionnumerics}

In this section we present two synthetic examples to illustrate the results discussed in Theorems \ref{coroptimalnon} and \ref{coroptimal}, regarding the Heavy ball system \eqref{DSGD}. In particular we test the performance of the trajectory generated by \eqref{DSGD}, in terms of objective function values \(F(x(t))-F_{\ast}\), with different values for $\alpha$ and the first-order in time Gradient flow system. All the experiments were performed in the Matlab programming language, by using the ODE solver "ode45" with absolute tolerance $\approx 10^{-13}$.

\paragraph{Example 1 (Convex, $\mu$-\eqref{KL})}
In the first example we consider a quadratic function of the form  $F(x)=\psc{Ax}{x}$, where\(A\in \R^{d\times d}\), \(x\in \R^d\) and $d=100$. $A$ is a random symmetric matrix, with its smallest eigenvalue equal to $0$ and some given smallest (and larger) positive eigenvalue $\mu$ (and $L$ respectively). With these choices, $F$ is a convex but not strongly convex function with $L$-Lipschitz gradient and satisfies the \eqref{KL} condition with parameter $\mu$.   We fix $L=1$ and make $3$ different choices for $\mu$, corresponding to $\kappa=\frac{L}{\mu}=10$, $\kappa=100$ and $\kappa=200$. In this example we compare the performance of the Gradient flow and the Heavy ball scheme \eqref{DSGD} with $4$ different choices for the damping parameter $\alpha$. One is set equal to $\alpha_{\ast}=(2\sqrt{\kappa}-\sqrt{\kappa-1})\sqrt{\mu}$ as Theorem \ref{coroptimal} indicates, two are slight variations of $\alpha_{\ast}$, i.e. $\alpha_{\pm}=\alpha_{\ast}\pm \varepsilon$, with $\varepsilon=0.1$ and the last on is chosen equal to $2\sqrt{\mu}$ (see e.g. \cite{polyak1964some,polyak2017lyapunov,haraux1998convergence}). The results are reported in Figure \ref{figQuadratic}. While the choice $\alpha=2\sqrt{\mu}$ has better performance in the well-conditioned case ($\kappa=10$), it takes more iterations to reach a smaller error than the one for the choices $\alpha_{\ast}$ (or $\alpha_{\ast} \pm\varepsilon$) in the ill conditioned cases ($\kappa=100$ or $\kappa=200$). In addition, we observe that the case $\alpha=(2\sqrt{\kappa}-\sqrt{\kappa-1})\sqrt{\mu}$ leads to less oscillatory behavior for $F(x(t))-F_{\ast}$, hinging a corresponding trajectory which overshoots less the set of minimizers $X_{\ast}$.

\begin{figure}[ht]
\centering
\includegraphics[scale=0.16,trim={6.9cm 2cm 2cm 2cm}]{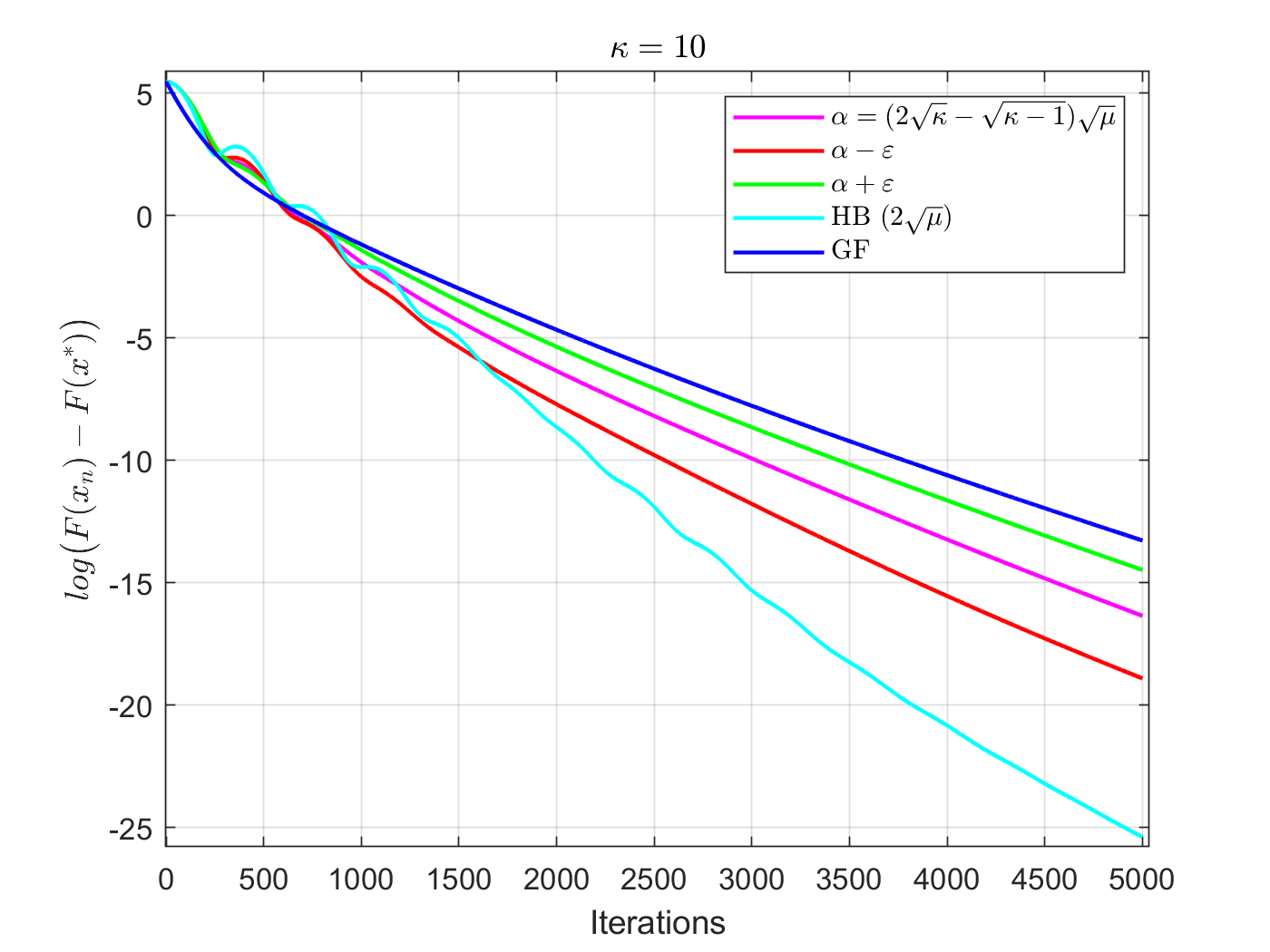} 
\includegraphics[scale=0.16,trim={1.4cm 2cm 2cm 2cm}]{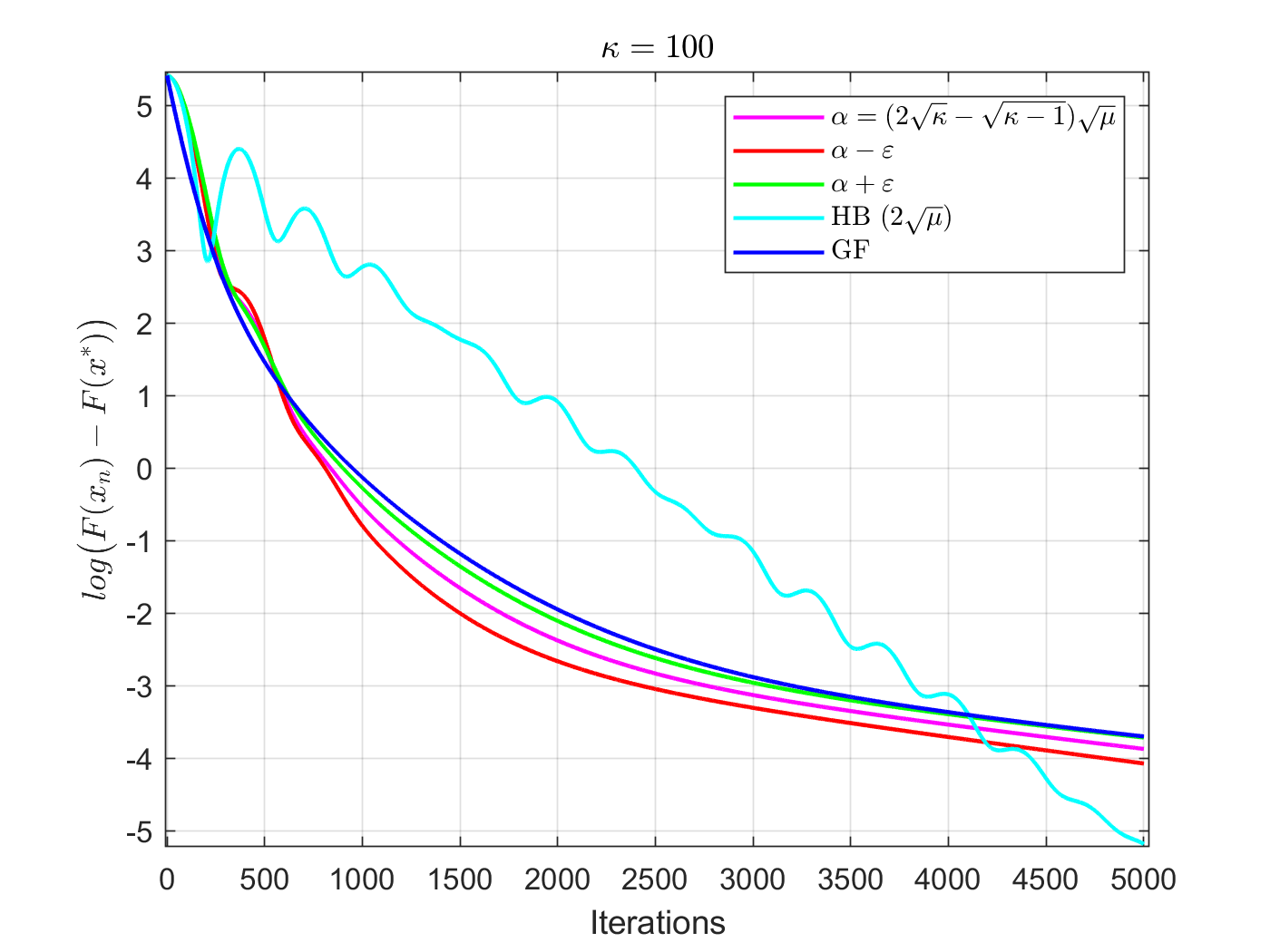}
\includegraphics[scale=0.16,trim={1.4cm 2cm 6.8cm 2cm}]{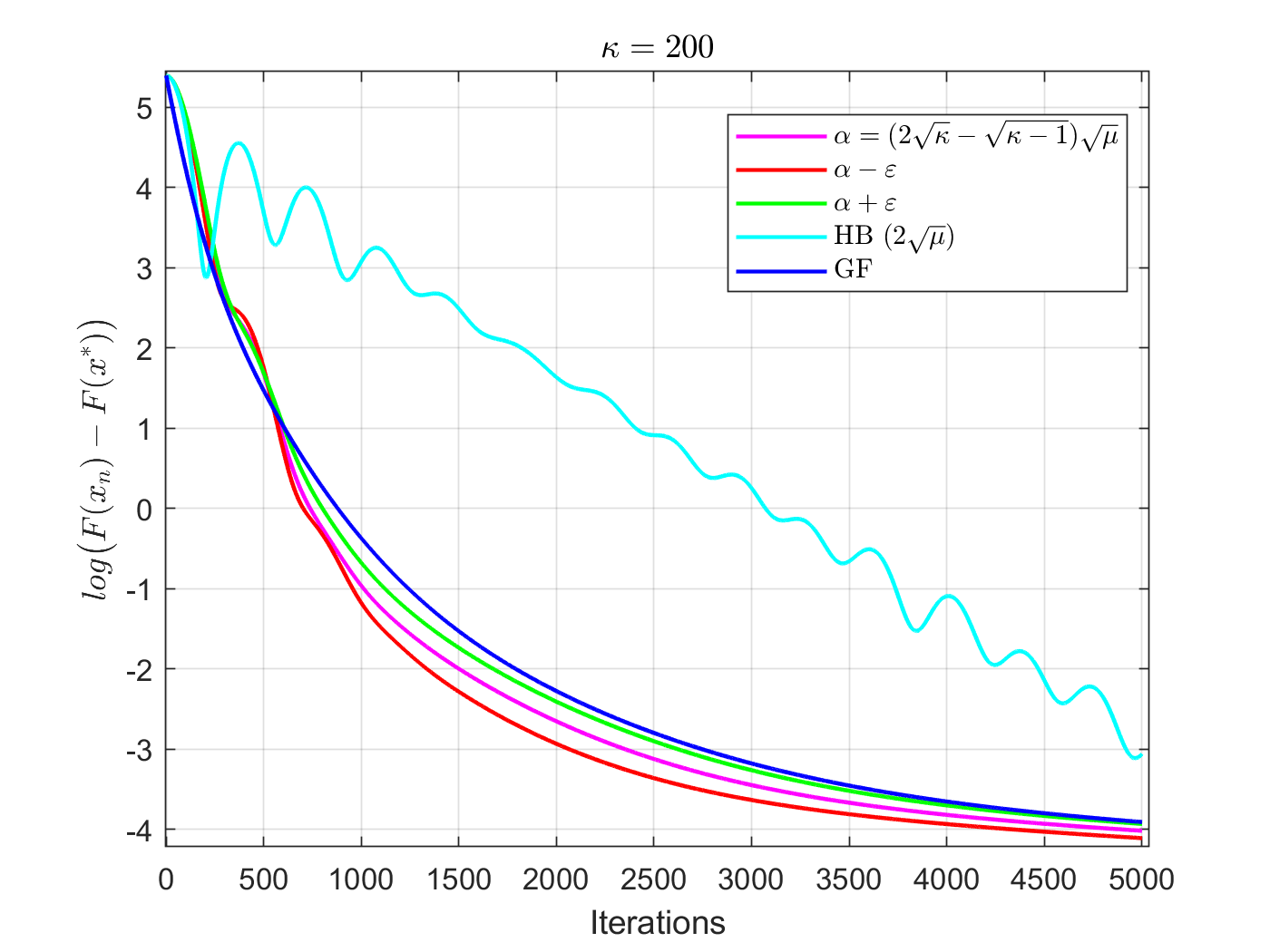}
\caption{Comparison of the Heavy Ball scheme \eqref{DSGD} with different values of $\alpha$ and the gradient flow scheme in terms of objective function values, for minimizing the quadratic function \(F(x)=\psc{Ax}{x} ~, ~ x\in \R^d\). The blue line corresponds to the gradient flow system, while the others to the HB system \eqref{DSGD} with $\alpha=\alpha_{\ast}:=(2\sqrt{\kappa}-\sqrt{\kappa-1})\sqrt{\mu}$ (magenta), $\alpha=\alpha_{\ast}-0.1$ (red), $\alpha=\alpha_{\ast}+0.1$ (green) and $\alpha=2\sqrt{\mu}$ (light blue). 
In each image a new matrix is chosen randomly with different conditional number \(\kappa=\frac{L}{\mu}\), starting from $\kappa=10$ (well-conditioned) to $\kappa=200$ (ill-conditioned). }\label{figQuadratic}
\end{figure}

\paragraph{Example 2 (non-convex, \(\mu\)-\eqref{KL})}

In the second example we consider the function $F:\R^{2}\to \R_{+}$, such that \(F(x,y)=\frac{\abs{y-\sin(x)}^2}{8}\). As discussed in the introduction (see Figure \ref{figKLnonconvex}), the function $F$ is not convex, but satisfies the \eqref{KL} condition with $\mu=\frac{1}{4}$. Its set of minimizers is the whole curve \(X_{\ast}=\{(x,y)\in \R^{2} ~: ~ y=\sin(x)\}\). It is also worth mentioning, that while $\nabla F$ is not globally Lipschitz, it is $L$-Lipshitz on sub-level sets. Therefore the results stated in Theorem \ref{coroptimalnon} are applicable (see also Remark \ref{remlocalization}). In this case, in Figure \ref{figODEsin}, we can observe that the performance of Heavy ball in terms of the objective function values is similar both for $\alpha=2\sqrt{\mu}$ and $\alpha=(2\sqrt{\kappa}-\sqrt{\kappa-1})\sqrt{\mu}$ and better than the one of Gradient flow.

\begin{figure}[ht]
\centering
\includegraphics[scale=0.19,trim={8cm 0.9cm 1cm 2cm}]{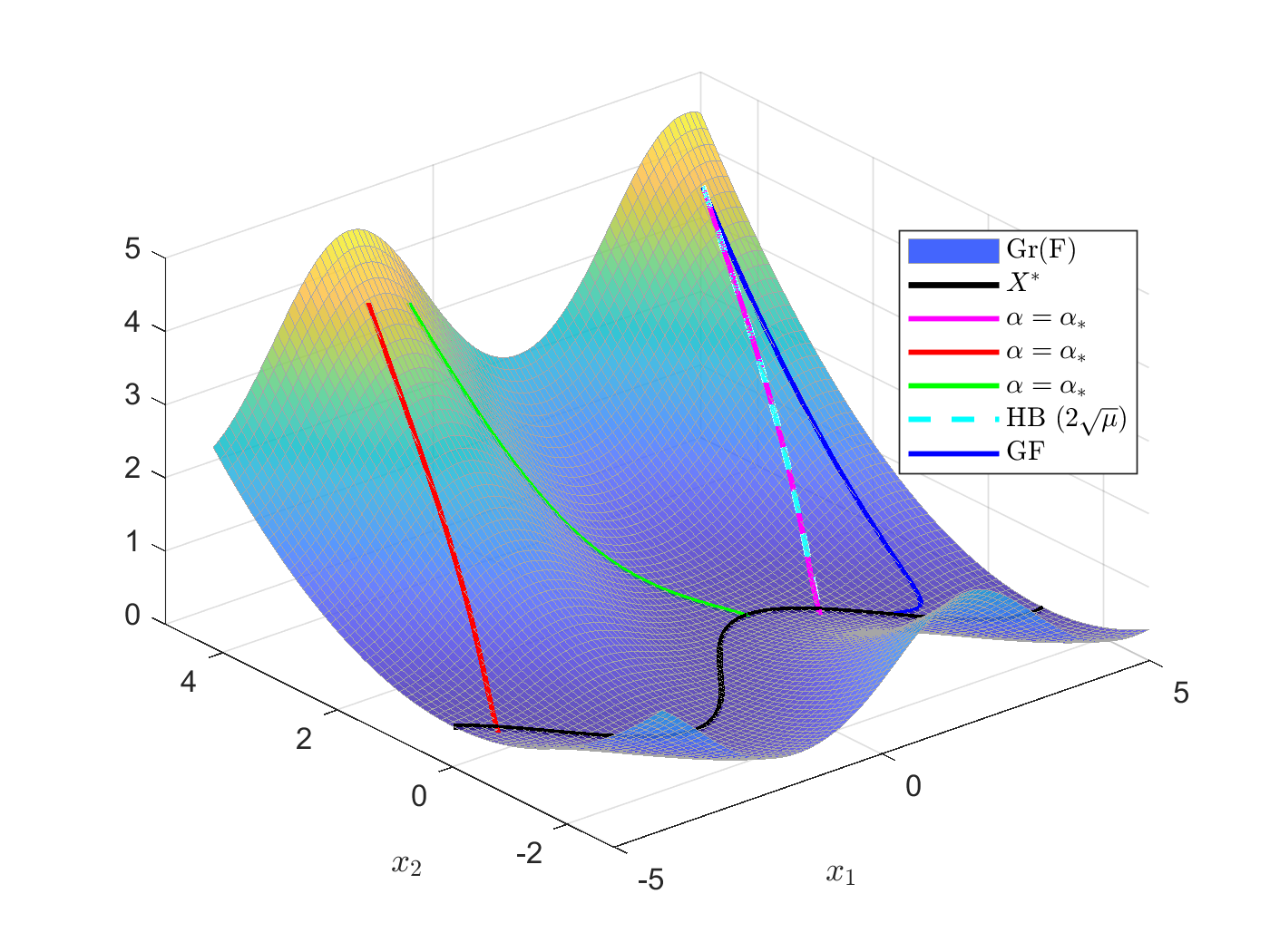} 
\includegraphics[scale=0.17,trim={1mm 0.4cm 1cm 3cm}]{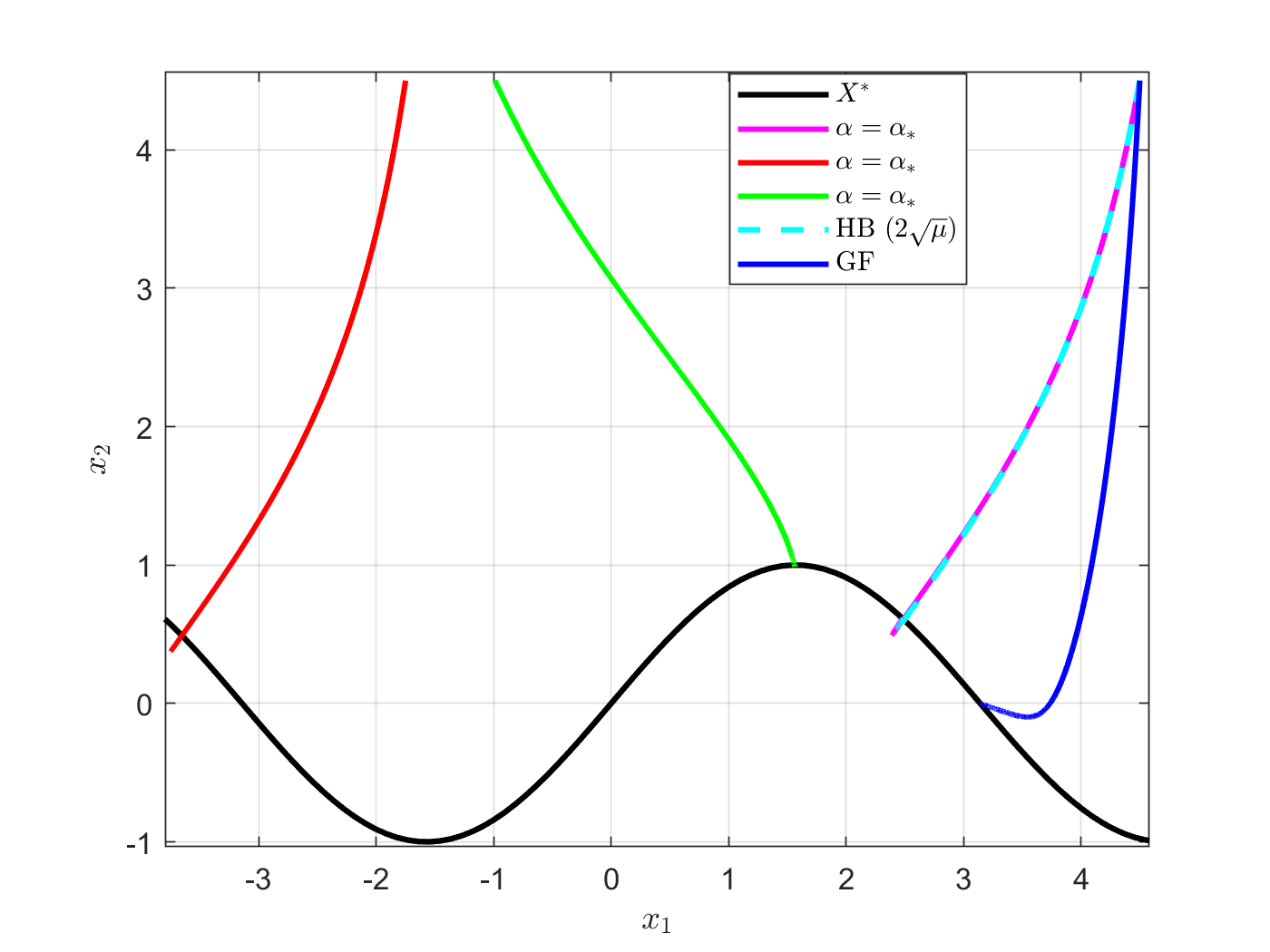}

\includegraphics[scale=0.2,trim={4.9cm 1cm 8mm 1mm}]{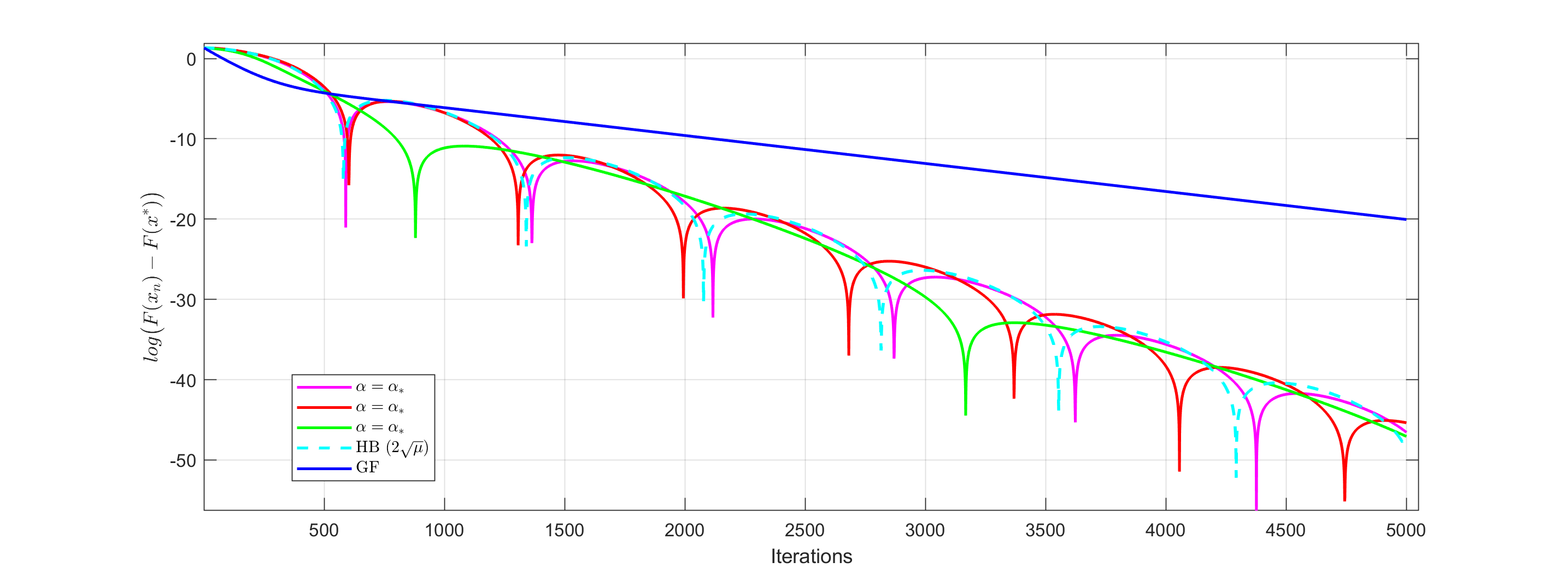}
\caption{In the first row we illustrate the convergence of trajectories of Gradient flow and the Heavy ball scheme with different starting points on the objective function landscape (left) and the $0$-level-set (right). The solid blue, solid magenta and dashed light blue represent the Gradient flow, Heavy ball with $\alpha=(2\sqrt{\kappa}-\sqrt{\kappa-1})\sqrt{\mu}$ and $\alpha=2\sqrt{\mu}$ respectively, with starting point set to $(4.5,4.5)$. The red and green line correspond to the Heavy ball with $\alpha=(2\sqrt{\kappa}-\sqrt{\kappa-1})\sqrt{\mu}$ and starting points $(-1.75,4.5)$ and $(-1,4.5)$ respectively.  All the initial velocities are set to $0$.
In the last row the corresponding objective function values are illustrated for each scheme.}\label{figODEsin}
\end{figure}

\section{Convergence analysis}\label{sectionConv}

Throughout this section, we consider a function $F:\Hb\to\R$ satisfying assumptions $\mathcal{A}.1$ and $\mathcal{A}.2$ and the Polyak-\L ojasiewicz condition \eqref{KL} and we denote $F_{\ast}=\min F$.

As mentioned before, the analysis that we follow in this Section in order to prove the main results in Theorems \ref{coroptimalnon} and \ref{coroptimal}, relies on Lyapunov techniques. The Lyapunov function \(V\) that we will introduce shortly was also used in previous works to study linear convergence of the trajectories of \eqref{DSGD} or other similar systems, see for example \cite{begout2015damped,polyak2017lyapunov,boct2018approaching}.

Let \(x(t)\) be a solution of \eqref{DSGD} and for sake of simplicity, we use the following notation for the objective error-function: 
\begin{equation}
W(t)=F(x(t))-F_{\ast}
\end{equation}
For some parameters \(a,\delta\in\R\) we define the following Lyapunov function for any $t\geq 0$:
\begin{equation}\label{Vdefinition}
\begin{aligned}
    V(t) &= aW(t)+\psc{\nabla F(x)}{\dot{x}(t)} + \frac{\delta}{2}\norme{\dot{x}(t)}^{2} \\
         &= aW(t)+\dot{W}(t) +\frac{\delta}{2}\norme{\dot{x}(t)}^{2}
\end{aligned}
\end{equation}
which will play a central role in our analysis.

For systems like \eqref{DSGD}, it is possible to choose other standard energy-functionals (see for example \cite{su2016differential,attouch2018convergence,apidopoulos2018differential,aujol2020HBconvergence}) which can be used to deduce interesting convergence properties for $x(t)$, which however usually require uniqueness of the minimizer. Regarding this issue, it is important to stress that the choice of $V$ does not depend explicitly on any minimizer of the objective function \(F\). On the other hand, from the forthcoming analysis it transpires that the Lipschitz character of the gradient cannot be avoided with this choice of energy-function.

In order to improve readability of the forthcoming analysis, we briefly sketch the salient steps, which are divided in three paragraphs. In paragraph \ref{subsection lyapunov} we provide some basic estimates for the energy \(V\), as also for $W$.
First, in Lemma \ref{lemma:expV} we show that under simple conditions on the parameters \(a\) and \(\delta\) and the damping parameter \(\alpha\), the Lyapunov function $V$ decreases linearly thanks to the hypothesis \eqref{KL}.
In Theorem \ref{basicteorates} in paragraph \ref{subsection non-convex} (see also Lemma \ref{lemmaexponentialforW}) we proceed by showing that the linear convergence of \(V\) implies the linear convergence of the objective function \(W\).

In the last part of paragraph \ref{subsection non-convex} we give the full proof of Theorem \ref{coroptimalnon}, by finding the optimal choices for the parameters $a$, $\delta$ and the damping coefficient $\alpha$, in order to have the fastest decay rate according to the Lyapunov analysis.

Finally, in paragraph \ref{subsection convex}, a similar sequence of proofs is carried out under the additional hypothesis of convexity of the objective function \(F\).

\subsection{Lyapunov estimates}\label{subsection lyapunov}

We start by proving that the Lyapunov function $V$ converges linearly to zero whenever \(t\to +\infty\).
\begin{lemma}\label{lemma:expV}
Let $F:\Hb\to\R$ satisfying $\mathcal{A}.1$, $\mathcal{A}.2$ and condition \eqref{KL}. Let also \(V\) be the function defined in \eqref{Vdefinition} and let:
\begin{equation}\label{eq:genericR}
	R \coloneqq \alpha - a + \delta. 
\end{equation}
Then, for any \(a,\delta\in\R_{+}\) satisfying
\begin{align}
&R > 0                                              \label{HR}\tag{HR} \\
	&aR \leq 2\mu 									\label{H1}\tag{H1} \\
	&L -\alpha\delta + \frac{\delta}{2}R \leq 0 	\label{H2}\tag{H2}
\end{align}
we have the following upper bound
\begin{equation}\label{eq:expV}
	V(t) \leq a(F(x_0)-F_{\ast})e^{-Rt}
\end{equation}
\end{lemma}
\begin{proof}
By a direct computation of the derivative of \(V(t)\) we get
\begin{equation}\label{prooflyap1}
    \dot{V}(t) = a\dot{W}(t) + \left\langle\frac{d}{dt}\nabla F(x(t)), \dot{x}(t)\right\rangle + \langle\nabla F(x(t)), \ddot{x}(t)\rangle + \delta\langle\dot{x}(t),\ddot{x}(t)\rangle
\end{equation}
Since \(\nabla F\) is $L$-Lipschitz continuous and $\dot{x}(\cdot)$ is continuous, it follows that $\nabla F(x(t))$ is absolutely continuous (see Remark $1$ in \cite{boct2018approaching}) and for almost every $t>0$, it holds:
\begin{equation}
	\left\langle\frac{d}{dt}\nabla F(x(t)), \dot{x}(t)\right\rangle \leq L\norme{\dot{x}(t)}^2
\end{equation}
Thus, from \eqref{prooflyap1}, we obtain:
\begin{equation}\label{prooflyap2}
    \dot{V}(t) \leq a\dot{W}(t) + L\norme{\dot{x}(t)}^2 + \langle\nabla F(x(t)), \ddot{x}(t)\rangle + \delta\langle\dot{x}(t),\ddot{x}(t)\rangle \quad  \text{ for a.e. } t\geq0
\end{equation}
Next, by using the basic equation \eqref{DSGD} for the solution $x(t)$ and expressing $\ddot{x}(t)=-\alpha\dot{x}(t)-\nabla F(x(t))$ in \eqref{prooflyap2}, we find:
\begin{equation}
\begin{aligned}
    \dot{V}(t) &\leq a\dot{W}(t) -\big(\alpha+\delta\big)\langle\nabla F(x(t)), \dot{x}(t)\rangle -\norme{\nabla F(x(t))}^2 + \big(L-\alpha\delta\big)\norme{\dot{x}(t)}^2 \quad \text{ for a.e. } t\geq0
\end{aligned}
\end{equation}
Since $\dot{W}(t)=\langle\nabla F(x(t)), \dot{x}(t)\rangle$, we have:
\begin{equation}
    \dot{V}(t) \leq  \big(a-\alpha-\delta\big)\dot{W}(t) -\norme{\nabla F(x(t))}^2 + \big(L-\alpha\delta\big)\norme{\dot{x}(t)}^2 \quad \text{ for a.e. } t\geq0
\end{equation}
and therefore, letting \(R=\alpha+\delta-a\) as in Equation~\eqref{eq:genericR}, we obtain
\begin{equation}\label{lyapunovcalculus1}
    \dot{V}(t) \leq -R\dot{W}(t) - \norme{\nabla F(x(t))}^2 + \big(L - \alpha\delta\big)\norme{\dot{x}(t)}^2 \quad \text{ for a.e. } t\geq0
\end{equation}
By definition of $V(t)$ \eqref{Vdefinition}, we have 
\begin{equation}\label{prooflyap3}
 \dot{W}(t)=V(t)-aW(t)-\frac{\delta}{2}\norme{\dot{x}(t)}^2
\end{equation}
hence by injecting \eqref{prooflyap3} into \eqref{lyapunovcalculus1}, we find:
\begin{equation}
	\dot{V}(t) \leq -RV(t) + aRW(t) - \norme{\nabla F(x(t))}^2 + \left(L - \alpha\delta + \frac{\delta}{2}R\right)\norme{\dot{x}(t)}^2 \quad \text{ for a.e. } t\geq0
\end{equation}
Since \(F\) satisfies condition \eqref{KL} (i.e. $2\mu W(t) \leq \norme{\nabla F((t))}^{2} $), we obtain
\begin{equation}
	\dot{V}(t) \leq -RV(t) + \left(aR - 2\mu\right)W(t) + \left(L -\alpha\delta + \frac{\delta}{2}R\right)\norme{\dot{x}(t)}^2 \quad \text{ for a.e. } t\geq0
\end{equation}
Conditions~\eqref{H1} and~\eqref{H2} yield:
\begin{equation}
	\dot{V}(t) \leq -RV(t) \quad \text{ for a.e. } t\geq0
\end{equation}
By direct integration we get
\begin{equation}\label{Vcontinuous}
	V(t) \leq V(0)e^{-Rt} \quad \text{ for a.e. } t\geq0
\end{equation}
and since $V$ is continuous, \eqref{Vcontinuous} holds true for all $t\geq0$.
Finally, the initial conditions \(x(0)=x_0\) and \(\dot{x}(0)=0\) in~\eqref{DSGD} imply \(V(0)=a(F(x_0)-F_{\ast})\).
\end{proof}

The next Lemma provides some sufficient conditions that ensure the validity of assumptions of Lemma \ref{lemma:expV} (i.e. conditions \eqref{HR}, \eqref{H1} and \eqref{H2}). 

\begin{lemma}\label{lemma:Hconditions}
Let $\delta>0$ and consider the following conditions:
\begin{equation}\label{eq:alpha-interval}
    \alpha\in\left(\frac{L}{\delta},\alpha_{-}\right]\cup\left[\alpha_{+},\delta+\frac{2L}{\delta}\right)
\end{equation}
where
\begin{equation}\label{eq:alphapm}
    \alpha_\pm = \frac12\left(\delta+\frac{3L}{\delta}\pm\sqrt{\left(\delta+\frac{L}{\delta}\right)^2-4\mu}\right)
\end{equation}
and
\begin{equation}\label{conditiona}
a = \delta +\frac{2L}{\delta} - \alpha
\end{equation}
Then the conditions \eqref{HR},\eqref{H1} and \eqref{H2} of Lemma \ref{lemma:expV} are satisfied, with
\begin{equation}
    R = 2\left(\alpha-\frac{L}{\delta}\right) 
\end{equation}
\end{lemma}

\begin{proof}

Let $\delta>0$ and $a=\delta+\frac{2L}{\delta}-\alpha$ in the definition of $V$ \eqref{Vdefinition}. 

Since $\alpha<\delta+\frac{2L}{\delta}$ we have $a>0$. By definition of $R$ in \eqref{eq:genericR} and $a=\delta+\frac{2L}{\delta}-\alpha$, we have $R=\alpha+\delta-a=2\left(\alpha-\frac{L}{\delta}\right)$ and since $\alpha>\frac{L}{\delta}$, it follows that $R>0$ and condition \eqref{HR} is satisfied. In addition since $R=2\left(\alpha-\frac{L}{\delta}\right)$,  condition \eqref{H2} holds true as equality.

Finally since $a=\delta+\frac{2L}{\delta}-\alpha$ and $R=2\left(\alpha-\frac{L}{\delta}\right)$ condition \eqref{H1} is equivalent to:
\begin{equation}\label{alphasecond}
    \alpha^2 - \left(\delta+\frac{3L}{\delta}\right)\alpha + 2\frac{L^2}{\delta^2} + L + \mu \geq 0
\end{equation}
The associated second order equation admits two real valued roots as defined in Equation~\eqref{eq:alphapm} for every \(\delta>0\), therefore the inequality \eqref{alphasecond} holds true if and only if \(\alpha\in[0,\alpha_-]\cup[\alpha_+,+\infty)\). Since the feasible set for $\alpha$ defined in Equation~\eqref{eq:alpha-interval} is contained in \(\alpha\in[0,\alpha_-]\cup[\alpha_+,+\infty)\), it follows that condition \eqref{H1} is satisfied and concludes the proof of Lemma \ref{lemma:Hconditions}.
\end{proof}

The next Lemma shows how we can take advantage of the linear rates for $V$, stated in Lemma \ref{lemma:expV}, to deduce linear rates for the objective function values $F(x(t))-F_{\ast}$.

\begin{lemma}\label{lemmaexponentialforW}
Let $F:\Hb\to\R$ satisfying $\mathcal{A}.1$, $\mathcal{A}.2$ and condition \eqref{KL} and \(V\) as defined in Equation~\eqref{Vdefinition}. Assume that conditions \eqref{HR}, \eqref{H1} and \eqref{H2} of Lemma~\ref{lemma:expV} hold true.
Then, for all $t>0$, we have:
\begin{equation}
F(x(t))-F_{\ast}\leq C(t)e^{-mt}
\end{equation}
where \(m =\min\{a,R\}\) and
\begin{equation}\label{eq:C}
	C(t) = \begin{dcases}
		\big(F(x(0))-F_{\ast}\big)\left(1+\frac{a}{\abs{a-R}}\right) &\text{if } a\neq R \\
		\big(F(x(0))-F_{\ast}\big)\left(1+at\right) 					&\text{if } a= R
	\end{dcases}
\end{equation}
\end{lemma}

\begin{proof}
By Lemma~\ref{lemma:expV} and Equation~\eqref{Vdefinition} we immediately get
\begin{equation}
\dot{W}(t) \leq -a W(t) +a W(0)e^{-Rt} -\frac{\delta}{2}\norme{\dot{x}(t)}^{2}
\end{equation}
Neglecting the non-negative term $\frac{\delta}{2}\norme{\dot{x}(t)}^{2}$, we obtain:
\begin{equation}
\dot{W}(t)\leq -aW(t) + a W(0)e^{-Rt}
\end{equation}
Hence, by applying Lemma \ref{gronwallcontinuous} (with $s=0$), we have:
\begin{equation}\label{lastinequality}
W(t)\leq W(0)e^{-at} + a W(0)e^{-at}\int_{0}^{t}e^{(a-R)s}ds
\end{equation}
By computing the integral in the last inequality \eqref{lastinequality} and neglecting the non-positive terms in the case $a\neq R$, we find:
\begin{equation}
W(t)\leq W(0)\left(e^{-at} +\frac{a}{\abs{a-R}}e^{-Rt}\right) \leq W(0)\left(1+\frac{a}{\abs{a-R}} \right)e^{-mt}
\end{equation}
with $m=\min\{a,R\}$.
On the other hand, if $a=R$, by computing the integral in \eqref{lastinequality}, we have:
\begin{equation}
W(t)\leq W(0)\big(1+at \big)e^{-mt}
\end{equation}
which concludes the proof of Lemma~\ref{lemmaexponentialforW}.
\end{proof}

\subsection{The non-convex setting}\label{subsection non-convex}

In the following Theorem we give a more explicit expression of the convergence rate of the objective function $F(x(t))-F_{\ast}$, depending on the damping parameter $\alpha$ and the auxiliary variable $\delta$. 

\begin{teo}\label{basicteorates}
Let $F:\Hb\to\R$ satisfying $\mathcal{A}.1$, $\mathcal{A}.2$ and condition \eqref{KL} with $\mu>0$ and denote $\kappa=\frac{L}{\mu}$. Let $\delta>0$ and $\big(x(t)\big)_{t\geq0}$ be the solution-trajectory associated to the dynamical system \eqref{DSGD}. Then for all \(t\geq 0\) the following bound holds
\begin{equation}\label{basicteoratesrelation}
F(x(t))-F_{\ast} \leq \big(F(x(0)-F_{\ast}\big)\bigg(1 + \frac{\delta+2\frac{L}{\delta}-\alpha}{\left|\delta+4\frac{L}{\delta}-3\alpha\right|}\bigg)e^{-mt} 
\end{equation}
with $m=\min\left\{\delta+2\frac{L}{\delta}-\alpha,2\big(\alpha-\frac{L}{\delta}\big)\right\}$, $\alpha\in \left(\frac{L}{\delta},\alpha_{-}\right]\cup\left[\alpha_{+},\delta+\frac{2L}{\delta}\right)$ where $\alpha_{\pm}=\frac{1}{2}\bigg(\delta+\frac{3L}{\delta}\pm\sqrt{\bigg(\delta+\frac{L}{\delta}\bigg)^2-4\mu}\bigg)$.

In particular we have the following cases:
\begin{itemize}
    \item \(m=2\left(\alpha-\frac{L}{\delta}\right)\) in the following cases:
    \begin{equation}\label{conditions11}
        \begin{dcases}
        \alpha\in\left(\frac{L}{\delta},\frac13\left(\delta+\frac{4L}{\delta}\right)\right) ~ \text{ if } ~ \kappa\leq \frac98\text{ and } \delta\in(\delta_-,\delta_+)   \\ \\
        \alpha\in\left(\frac{L}{\delta},\alpha_-\right] \\ \quad  \text{ if } ~ \biggl(\kappa\leq\frac98 \text{ and } \delta\in(0,\delta_-)\cup(\delta_+,+\infty)\biggr) \text{ OR } \biggl(\kappa\geq\frac98 \text{ and } \delta>0\biggr)
        \end{dcases}
    \end{equation}
    \item \(m=\delta+2\frac{L}{\delta}-\alpha\) in the following cases:
    \begin{equation}\label{conditions12}
    \begin{dcases}
         \alpha \in \left(\frac13\left(\frac{4L}{\delta}+\delta\right), \alpha_-\right]\cup\bigl[\alpha_+,\delta+\frac{2L}{\delta}\bigr) ~ \text{ if } ~ \kappa\leq \frac98\text{ and } \delta\in(\delta_-,\delta_+)\  \\
         \\
        \alpha \in [\alpha_+,\delta+\frac{2L}{\delta})  \\ \quad \text{ if } ~ \biggl(\kappa\leq\frac98 \text{ and } \delta\in(0,\delta_-)\cup(\delta_+,+\infty)\biggr) \text{ or } \biggl(\kappa\geq\frac98 \text{ and } \delta>0\biggr) 
    \end{dcases}
    \end{equation}
\end{itemize}
where \(\delta_\pm = \frac{1}{2\sqrt{2}}\left(3\sqrt{\mu}\pm\sqrt{9\mu-8L}\right)\).
\end{teo}

\begin{proof} 

Let $\delta>0$, $\alpha\in \left(\frac{L}{\delta},\alpha_{-}\right]\cup\left[\alpha_{+},\delta+\frac{2L}{\delta}\right)$ and set $a=\delta+\frac{2L}{\delta}-\alpha$ in the definition of $V$ \eqref{Vdefinition}, so that Lemmas~\ref{lemma:expV}, and \ref{lemma:Hconditions} hold true with $R=2\left(\alpha-\frac{L}{\delta}\right)$.

By using Lemma~\ref{lemmaexponentialforW} it follows that for all $t>0$, it holds:
\begin{equation}
F(x(t))-F_{\ast}\leq C(t)e^{-mt}
\end{equation}
with \(m =\min\left\{a,R\right\}\), with $a=\delta+\frac{2L}{\delta}-\alpha$ and $R=2\big(\alpha-\frac{L}{\delta}\big)$.
In addition if we exclude the case $a=R$ (i.e. $\alpha=\frac{\delta+\frac{4L}{\delta}}{3}$),  Lemma~\ref{lemmaexponentialforW}, implies:
\[
    C(t) = 1 + \frac{a}{\left|a-R\right|}
\]
By substituting \(a=\delta+2\frac{L}{\delta}-\alpha\) and \(R=2\left(\alpha-\frac{L}{\delta}\right)\), we obtain
\[
    C = 1 + \frac{\delta+2\frac{L}{\delta}-\alpha}{\left|\delta+4\frac{L}{\delta}-3\alpha\right|}
\]

To make more explicit the rates of the convergence of the objective function we need to study the value of \(m\), as a function of $\alpha$ and $\delta$.

Recalling that \(\alpha\in\left(\frac{L}{\delta},\alpha_-\right]\cup\left[\alpha_+,\delta+\frac{2L}{\delta}\right)\) where
\[
    \alpha_\pm = \frac12\left(\frac{3L}{\delta}+\delta \pm \sqrt{\left(\frac{L}{\delta} + \delta\right)^2-4\mu}\right)
\]
we consider the following two cases:
\begin{itemize}
    \item \(m = 2\left(\alpha-\frac{L}{\delta}\right)\) 
    
    This case is equivalent to the condition 
    \[
        \delta+2\frac{L}{\delta}-\alpha > 2\left(\alpha-\frac{L}{\delta}\right)
    \]
which can be rewritten as 
    \[
        \alpha < \frac13\left(\delta+\frac{4L}{\delta}\right)
    \]
    First we note that we always have \(\frac13\left(\delta+\frac{4L}{\delta}\right)\leq \alpha_+\).
    Indeed, by a direct computation
    \begin{align*}
        \frac13\left(\delta+\frac{4L}{\delta}\right) &\leq \frac12\left(\frac{3L}{\delta}+\delta + \sqrt{\left(\frac{L}{\delta} + \delta\right)^2-4\mu}\right) \\
        -\frac13\left(\frac{L}{\delta}+\delta\right) &\leq \sqrt{\left(\frac{L}{\delta} + \delta\right)^2-4\mu}
    \end{align*}
    which is true for every \(L\), \(\mu\) and \(\delta\).
    
Therefore we have \(\alpha\in\left(\frac{L}{\delta},\min\{\alpha_-, \frac13\left(\delta+\frac{4L}{\delta}\right)\}\right)\).
    We start ordering the values in the minimum by studying the case \(\frac13\left(\delta+\frac{4L}{\delta}\right)< \alpha_-\) (note that we exclude the equality since we assumed that \(\alpha\) cannot take the value \(\frac13\left(\delta+\frac{4L}{\delta}\right)\))  which is equivalent to
    \begin{align*}
        \frac13\left(\delta+\frac{4L}{\delta}\right)    &< \frac12\left(\frac{3L}{\delta}+\delta - \sqrt{\left(\frac{L}{\delta} + \delta\right)^2-4\mu}\right) \\
        \frac13\left(\frac{L}{\delta}+\delta\right)     &> \sqrt{\left(\frac{L}{\delta} + \delta\right)^2-4\mu}
    \end{align*}
    By taking the square we obtain
    \begin{equation} \label{eq:delta_roots}
        \left(\frac{L}{\delta} + \delta\right)^2 < \frac92\mu
    \end{equation}
    or equivalently
    \[
        \delta^2 - \frac{3}{\sqrt{2}}\sqrt{\mu}\delta + L < 0
    \]
    which is satisfied if \(\kappa=\frac{L}{\mu}\leq\frac98\) and \(\delta\in(\delta_-,\delta_+)\), where
    \[
        \delta_\pm = \frac{1}{2\sqrt{2}}\left(3\sqrt{\mu}\pm\sqrt{9\mu-8L}\right)
    \]
    Analogously, we find that \(\frac13\left(\delta+\frac{4L}{\delta}\right) > \alpha_-\) can be reduced to
    \[
        \delta^2 - \frac{3}{\sqrt{2}}\sqrt{\mu}\delta + L > 0
    \]
    which holds true if  \(\kappa=\frac{L}{\mu}\leq\frac98\) and \(\delta\in(0,\delta_-)\cup(\delta_+,+\infty)\)   or if \(\kappa\geq\frac98\) and \(\delta>0\)  .
    \item \(m = \delta+2\frac{L}{\delta}-\alpha\) 
    
    This case can be studied analogously to the previous one by noticing that this value of \(m\) is achieved whenever
    \[
        \delta+2\frac{L}{\delta}-\alpha > 2\left(\alpha-\frac{L}{\delta}\right)
    \]
    which is equivalent to
    \[
        \alpha > \frac13\left(\delta+\frac{4L}{\delta}\right)
    \]
    together with \(\alpha\in\left(\frac{L}{\delta},\alpha_-\right]\cup[\alpha_+,\delta+\frac{2L}{\delta})\).
   The statement follows from the relation between \(\frac13\left(\delta+\frac{4L}{\delta}\right)\) and \(\alpha_\pm\) as already studied in the previous point.
    \end{itemize}
\end{proof}

We are now ready to give the proof of Theorem~\ref{coroptimalnon}.

\begin{proof}[\textbf{Proof of Theorem \ref{coroptimalnon}}]

To prove the statement of Theorem we optimize (maximize) the rate $m=\min\left\{\delta+2\frac{L}{\delta}-\alpha,2\big(\alpha-\frac{L}{\delta}\big)\right\}$ over the feasible values for \(\alpha\) and \(\delta\), defined in Theorem~\ref{basicteorates}. In particular, by \eqref{conditions11} and \eqref{conditions12}, we consider two cases:
\begin{itemize}
    \item \(m=2\left(\alpha-\frac{L}{\delta}\right)\) 
    
    In this case by maximizing $m$, over $\alpha$ and $\delta$, if $\varepsilon>0$, we have:
 
    \begin{equation}
    \begin{aligned}
        m^* & = 2\sup_{\alpha, \delta}\left(\alpha-\frac{L}{\delta}\right) \\ &=
        \begin{dcases}
            \sup_{\delta\in[\delta_-, \delta_+]}\frac23\left(\frac{L}{\delta}+\delta\right)-\varepsilon    & \text{if}\ \kappa\leq\frac98 ~ \text{ and } ~ \alpha=\frac{\delta+\frac{4L}{\delta}}{3}-\frac{\varepsilon}{2} \\
            \sup_{\delta\in(0,\delta_-]\cup[\delta_+,+\infty)}2\left(\alpha_- - \frac{L}{\delta}\right)-\varepsilon & \text{if}\ \kappa\leq\frac98 ~ \text{ and } ~ \alpha=\alpha_{-}-\frac{\varepsilon}{2} \\
            \sup_{\delta>0}2\left(\alpha_- - \frac{L}{\delta}\right) & \text{if}\ \kappa>\frac98 ~ \text{ and } ~ \alpha=\alpha_{-}
        \end{dcases}
        \end{aligned}
        \end{equation}
 where $\alpha_{-}=\frac{1}{2}\left(\delta+\frac{3L}{\delta}-\sqrt{\left(\delta+\frac{L}{\delta}\right)^2-4\mu}\right)$.
        
For the first case, note that the supremum of \(\frac{L}{\delta}+\delta\) at $[\delta_{-},\delta_{+}]$ is attained at $\delta=\delta_{\pm}$ and is equal to \(m^*=\sqrt{2\mu}-\varepsilon\).
    
For the second case, since \(2\left(\alpha_- - \frac{L}{\delta}\right)-\varepsilon=\delta+\frac{L}{\delta}-\sqrt{\left(\delta+\frac{L}{\delta}\right)^2-4\mu}-\varepsilon\), its maximum value for \(\delta\in(0,\delta_-]\cup[\delta_+,+\infty)\) is achieved for \(\delta=\delta_\pm\) and is equal to \(m^*=\sqrt{2\mu}-\varepsilon\).

Finally, in the third case the maximum for \(2\left(\alpha_- - \frac{L}{\delta}\right)\) for $\delta>0$, is achieved for \(\delta=\sqrt{L}\) and equals to \(m^*=2(\sqrt{L}-\sqrt{L-\mu})\).

Concluding, without loss of generality, we have the following possible rates:
\begin{equation}\label{eq:ma>r}
        m^* = \begin{dcases}
            \sqrt{2\mu}-\varepsilon & \text{ if } ~ \kappa\leq\frac98 ~ \text{ and } ~ \alpha =\frac{(5\pm\sqrt{9-8\kappa})}{2\sqrt{2}}\sqrt{\mu}-\frac{\varepsilon}{2} \\
            2\left(\sqrt{\kappa}-\sqrt{\kappa-1}\right)\sqrt{\mu} & \text{ if } ~ \kappa>\frac98 ~ \text{ and } ~ \alpha =\left(2\sqrt{\kappa}-\sqrt{\kappa-1}\right)\sqrt{\mu} \\
        \end{dcases}
    \end{equation}
    \item \(m = \delta+2\frac{L}{\delta}-\alpha\) 
    
    Analogously, in this case the best possible convergence rate is given by
    \begin{equation}\label{diorthoseis}
    \begin{aligned}
        m^* & = \sup_{\alpha, \delta}\left(\delta +\frac{2L}{\delta}-\alpha\right) \\ & = 
        \begin{dcases}
            \sup_{\delta\in(\delta_-, \delta_+)}\frac23\left[\frac{L}{\delta} + \delta\right] & \text{if }\ \kappa\leq\frac98 ~ \text{ and } ~ \alpha=\frac{\delta+\frac{4L}{\delta}}{3}-\varepsilon \\
            \sup_{\delta\in(0,\delta_-]\cup[\delta_+,+\infty)}\left(\delta +\frac{2L}{\delta}-\alpha_+\right) & \text{if }\ \kappa\leq\frac98 ~ \text{ and } ~ \alpha=\alpha_{+} \\
            \sup_{\delta\in\R}\left(\delta +\frac{2L}{\delta}-\alpha_+\right) & \text{ if }\ \kappa>\frac98 ~ \text{ and } ~ \alpha=\alpha_{+} 
        \end{dcases}
    \end{aligned}
    \end{equation}
  where $\alpha_{+}=\frac{1}{2}\left(\delta+\frac{3L}{\delta}+\sqrt{\left(\delta+\frac{L}{\delta}\right)^2-4\mu}\right)$ and $\varepsilon>0$.   
    
    In the first case of \eqref{diorthoseis} we have the same expression we studied in  the previous section which  gives \(m^*=\sqrt{2\mu}-\varepsilon\).
    For the second, since \(\delta +\frac{2L}{\delta}-\alpha_+=\frac{1}{2}\left(\delta+\frac{L}{\delta}-\sqrt{\left(\delta+\frac{L}{\delta}\right)^2-4\mu}\right)\) its maximum value in \(\sqrt{L}\notin(0,\delta_-]\cup[\delta_+,+\infty)\) is attained for \(\delta=\delta_\pm\)  and is \(m^*=\sqrt{\frac{\mu}{2}}\), while in the third case, the maximal value is achieved for \(\delta=\sqrt{L}\) and equals to \(m^*=\sqrt{L}-\sqrt{L-\mu}\).
Overall, in this case, we have
    \begin{equation}\label{eq:ma<r}
        m^* = \begin{dcases}
            \sqrt{2\mu}-\varepsilon    & \text{ if }\ \kappa\leq\frac98 \\
            \left(\sqrt{\kappa}-\sqrt{\kappa-1}\right)\sqrt{\mu}   & \text{ if }\ \kappa>\frac98 \\
        \end{dcases}
    \end{equation}
\end{itemize}

Comparing \eqref{eq:ma>r} and \eqref{eq:ma<r}, the best value for $m^{\ast}$ is:
\begin{equation}
        m^* = \begin{dcases}
            \sqrt{2\mu}-\varepsilon & \text{ if } ~ \kappa\leq\frac98 ~ \text{ and } ~ \alpha =\frac{(5\pm\sqrt{9-8\kappa})}{2\sqrt{2}}\sqrt{\mu}-\frac\varepsilon2 \\
            2\left(\sqrt{\kappa}-\sqrt{\kappa-1}\right)\sqrt{\mu} & \text{ if } ~ \kappa>\frac98 ~ \text{ and } ~ \alpha =\left(2\sqrt{\kappa}-\sqrt{\kappa-1}\right)\sqrt{\mu} \\
        \end{dcases}
    \end{equation}

Finally by computing the associated constant of relation \eqref{basicteoratesrelation} in Theorem \ref{basicteorates}, for each case, if $\varepsilon>0$, then it holds: 
\begin{equation}
C= \big(F(x(0)-F_{\ast}\big)\begin{dcases}   \frac{2\left(2\varepsilon+\sqrt{2\mu}\right)}{3\varepsilon}   & ~ \text{ if } ~ \kappa\leq\frac98 \\
\frac{4\left(3\kappa-3+\sqrt{\kappa}\sqrt{\kappa-1}\right)}{8\kappa-9} & ~ \text{ if } ~ \kappa>\frac98
\end{dcases}
\end{equation}
which concludes the proof of Theorem \ref{coroptimalnon}.
\end{proof}

\subsection{Convex setting}\label{subsection convex}

In this section we prove Theorem \ref{coroptimal} dealing with the convex setting. In particular we will get some slightly improved rates with respect to the ones found in the previous Theorem \ref{basicteorates}.
This improvement is reflected in the following Theorem.

\begin{teo}\label{basicteoratesconvex}
Let $F$ be a convex function with $L$-Lipschitz gradient, satisfying condition \ref{KL} with $\mu>0$ and $\big((x(t)\big)_{t\geq0}$ the solution-trajectory associated to the dynamical system \eqref{DSGD}.
Let $\delta>0$ and $\alpha_{-}=\frac{1}{2}\bigg(\delta+\frac{3L}{\delta}-\sqrt{\bigg(\delta+\frac{L}{\delta}\bigg)^2-4\mu}\bigg)$
We consider the following cases:
\begin{itemize}
\item If $\mu=L$ and $\delta\leq \sqrt{L}$, then \eqref{ratesObjectiveTeo} holds true for all $\alpha\in \big(\frac{L}{\delta},\alpha_{-}\big)$ 

\item If $\mu<L$ and  $\delta>0$ or if $\mu=L$ and  $\delta>\sqrt{L}$, then \eqref{ratesObjectiveTeo} holds true for all $\alpha\in \big(\frac{L}{\delta},\alpha_{-}\big]$
\end{itemize}
\begin{equation}\label{ratesObjectiveTeo}
\begin{aligned}
\norme{\nabla F(x(t))}^{2} & \leq C_{\delta,\alpha}\big(F(x(0))-F_{\ast}\big)e^{-2\left(\alpha-\frac{L}{\delta}\right)t} \\
\text{and } \qquad F(x(t))-F_{\ast} &\leq \frac{2C_{\delta,\alpha}}{\mu}\big(F(x(0)-F_{\ast}\big)e^{-2\big(\alpha-\frac{L}{\delta}\big)t}
\end{aligned}
\end{equation}
with $C_{\delta,\alpha}=2\bigg(1+\frac{L}{\delta\big(\delta+\frac{L}{\delta}-\alpha\big)}\bigg)$.

\end{teo}

\begin{proof} 
By using Cauchy–Schwarz inequality and then Young's inequality, for the scalar product $\psc{\nabla F(x(t))}{\dot{x}(t)}$,  for all $\eta>0$ and $t\geq 0$, we have:
\begin{equation}\label{Youngsproof}
   \psc{\nabla F(x(t))}{\dot{x}(t)}\geq -\frac{1}{2}\bigg(\frac{1}{\eta}\norme{\nabla F(x(t)}^2+\eta\norme{\dot{x}(t)}^2\bigg) 
\end{equation}

By injecting the previous inequality \eqref{Youngsproof} into the definition of the energy $V$ \eqref{Vdefinition}, we find: 
\begin{equation}\label{Vassiteleutaio}
a\big(F(x(t)-F_{\ast}\big)-\frac{1}{2}\bigg(\frac{1}{\eta}\norme{\nabla F(x(t)}^2+\eta\norme{\dot{x}(t)}^2\bigg) +\frac{\delta}{2}\norme{\dot{x}(t)}^2\leq V(t)    
\end{equation}
From convexity and the $L$-Lipschitz character of $\nabla F$ (see Proposition \ref{convexLipschitz} in Appendix) and \eqref{Vassiteleutaio}, it follows that for all $\eta>0$ and $t\geq0$, it holds:
\begin{equation}\label{lowerboundVt}
V(t)\geq \frac{1}{2}\bigg(\frac{a}{L}-\frac{1}{\eta}\bigg)\norme{\nabla F(x(t))}^{2}+\frac{1}{2}\bigg(\delta-\eta \bigg)\norme{\dot{x}(t)}^{2}
\end{equation}

Let us choose $\eta =\delta$, so that \eqref{lowerboundVt} becomes:

\begin{equation}\label{proofLipschitzconvex}
V(t)\geq \frac{1}{2}\bigg(\frac{a}{L}-\frac{1}{\delta}\bigg)\norme{\nabla F(x(t))}^{2}
\end{equation} 

Therefore, from \eqref{proofLipschitzconvex}, if $a>\frac{L}{\delta}$ and the conditions of Lemma \ref{lemma:expV} are satisfied, then it is immediate that for all $t\geq 0$, it holds:
\begin{equation}\label{proofratesnablaF}
\norme{\nabla F(x(t))}^{2}\leq \frac{2}{a-\frac{L}{\delta}}V(t)\leq \frac{2V(0)}{a-\frac{L}{\delta}}e^{-Rt}
\end{equation}
with $R=2\bigg(\alpha-\frac{L}{\delta}\bigg)$.

Next we show that the conditions of Lemma \ref{lemma:expV} are satisfied. In fact by Lemma \ref{lemma:Hconditions}, we recall that for all $\delta>0$, if we set $\alpha\in (\frac{L}{\delta},\alpha_{-}]\cup[\alpha_{+},\delta+\frac{2L}{\delta})$, with $\alpha_{\pm}=\frac{1}{2}\bigg(\delta+\frac{3L}{\delta}\pm\sqrt{\bigg(\delta+\frac{L}{\delta}\bigg)^2-4\mu}\bigg)$ and $a=\delta+\frac{2L}{\delta}-\alpha$, then Lemma \ref{lemma:expV} holds true with $R=2\left(\alpha-\frac{L}{\delta}\right)$. By imposing additionally that $a=\delta+\frac{2L}{\delta}-\alpha>\frac{L}{\delta}$, a straightforward computation shows that the damping parameter $\alpha$ should also satisfy $\alpha<\delta+\frac{L}{\delta}$.
Thus the overall conditions for $\alpha$, $\delta$ and $a$, so that \eqref{proofratesnablaF} holds true, are (here notice that $\delta+\frac{L}{\delta}\leq\alpha_{+}$):
\begin{equation}\label{casesall}
\begin{dcases}
\delta >0 \\
a=\delta+\frac{2L}{\delta}-\alpha \\
\alpha \in \bigg(\frac{L}{\delta},\min\big\{\delta+\frac{L}{\delta},\alpha_{-}\big\}\bigg) \quad , \quad \alpha_{-}=\frac{1}{2}\bigg(\delta+\frac{3L}{\delta}-\sqrt{\bigg(\delta+\frac{L}{\delta}\bigg)^2-4\mu}\bigg) 
\end{dcases}    
\end{equation}

Note that $\alpha_{-}\leq \delta+\frac{L}{\delta}$, holds always true and is saturated (holds as an equality), if and only if $L=\mu$ and $0<\delta\leq \sqrt{L}$.

In the case $\mu=L$, if $0<\delta\leq \sqrt{L}$, we have $\delta+\frac{L}{\delta}=\alpha_{-}$, thus from \eqref{casesall}, the rate is $R=2\bigg(\alpha-\frac{L}{\delta}\bigg)$, for all $\alpha\in \big(\frac{L}{\delta},\delta+\frac{L}{\delta}\big)$. 

If instead $\delta>\sqrt{L}$, then $\alpha_{-}<\delta+\frac{L}{\delta}$, thus from \eqref{casesall}, $R=2\bigg(\alpha-\frac{L}{\delta}\bigg)$ for all $\alpha\in \big(\frac{L}{\delta},\alpha_{-}\big)$.

In the case $\mu<L$, since $\alpha_{-}<\delta+\frac{L}{\delta}$ for all $\delta>0$, hence as before we have $R=2\bigg(\alpha-\frac{L}{\delta}\bigg)>0$, for all $\alpha\in \big(\frac{L}{\delta},\alpha_{-}\big)$.

Finally, for both cases from \eqref{proofratesnablaF}, with $a=\delta+\frac{2L}{\delta}-\alpha$ and $R=2\bigg(\alpha-\frac{L}{\delta}\bigg)$, we have:
\red{
\begin{equation}\label{gradientestimateconvex}
\norme{\nabla F(x(t))}^{2} \leq \frac{2\left(\delta+\frac{2L}{\delta}-\alpha\right)}{\delta+\frac{L}{\delta}-\alpha}\big(F(x(0))-F_{\ast}\big)e^{-2\left(\alpha-\frac{L}{\delta}\right)t}
\end{equation}
}
and  by using condition \eqref{KL} in \eqref{gradientestimateconvex}, we obtain:
\begin{equation}
F(x(t))-F_{\ast} \leq  \frac{\delta+\frac{2L}{\delta}-\alpha}{\mu(\delta+\frac{L}{\delta}-\alpha)}\big(F(x(0))-F_{\ast}\big)e^{-2\left(\alpha-\frac{L}{\delta}\right)t}
\end{equation}
which allows to conclude the proof of Theorem \ref{basicteoratesconvex}.
\end{proof}

\begin{remark}
In relation \eqref{ratesObjectiveTeo} of Theorem \ref{basicteoratesconvex}, the rate for $\norme{\nabla F(x(t))}^2$ and $F(x(t))-F_{\ast}$ is expressed as a function of both the damping term $\alpha$ and the auxiliary parameter $\delta$,  which also determines the choice of $\alpha$. The presence of $\delta$, is due to the Lyapunov analysis that we follow, where $\delta>0$ plays the role of a Lyapunov parameter (which is classically set free for tuning). Note that given any $\alpha$, there exists a $\delta$, such that the conditions of Theorem \ref{basicteoratesconvex} are always satisfied. 

Indeed, from \eqref{ratesObjectiveTeo} of Theorem \ref{basicteoratesconvex}, $\alpha$ can be chosen as follows:
\begin{equation}\label{alphaminus}
    \alpha=\alpha_{-}=\frac{1}{2}\bigg(\delta+\frac{3L}{\delta}-\sqrt{\bigg(\delta+\frac{L}{\delta}\bigg)^2-4\mu}\bigg)
\end{equation}
Since $\delta\to \alpha(\delta)$ is continuous and strictly decreasing, one could solve \eqref{alphaminus} for $\delta$, that is
\begin{equation}\label{thirdorderdelta}
    \begin{cases} & \alpha\delta^{3}-(\alpha^2+L+\mu)\delta^2+3L\alpha\delta -2L^2=0 \\ 
\text{s.t.: }    & \bigl(\alpha\geq \sqrt{3L})\text{ \& } \delta \in (0,\alpha-\sqrt{\alpha^2-3L})\cup(\alpha-\sqrt{\alpha^2-3L},+\infty)\bigr) \\ \text{ or } & \bigl((\alpha< \sqrt{L}) \text{ \& } \delta>0\bigr)
    \end{cases}
\end{equation}
Unfortunately the (positive) root of the third-order polynomial in \eqref{thirdorderdelta} cannot be given via an algebraic formulation, which is why the auxiliary variable $\delta>0$ is included in the statement of the Theorem.
\end{remark}

We are now ready to give the proof of Theorem \ref{coroptimal}.

\begin{proof}[\textbf{Proof of Theorem~\ref{coroptimal}}]
From Theorem \ref{basicteoratesconvex} and in particular relation \eqref{ratesObjectiveTeo}, we have $R=2\bigg(\alpha-\frac{L}{\delta}\bigg)$, hence it is of best interest to choose $\alpha$ as large as possible (with respect to $\delta$). 

Without loss of generality, let us consider the two complementary cases according to the Theorem \ref{basicteoratesconvex}:
\begin{itemize}

   \item Let $\mu=L$ and $\alpha\in \bigl(\frac{L}{\delta},\alpha_{-}\bigr)$, for all $\delta>0$. According to the first observation and taking the maximal value for $\alpha$, for any $\varepsilon \in (0,\alpha_{-}-\frac{L}{\delta})$, we can choose $\alpha=\alpha_{-}-\varepsilon=\frac{1}{2}\bigg(\delta+\frac{3L}{\delta}-\sqrt{\bigg(\delta-\frac{L}{\delta}\bigg)^2}\bigg)-\varepsilon=\frac{1}{2}\bigg(\delta+\frac{3L}{\delta}-\bigl|\delta-\frac{L}{\delta}\bigr|\bigg)-\varepsilon$. It follows that the convergence factor $R$ is equal to 
\begin{equation}\label{ratesRfinal}
R=R(\delta)=2\bigg(\alpha_{-}-\varepsilon-\frac{L}{\delta}\bigg) = \delta +\frac{L}{\delta} -\biggl|\delta-\frac{L}{\delta}\biggr| -2\varepsilon \qquad \forall \delta>0
\end{equation}

Therefore by studying the expression of $R$ in \eqref{ratesRfinal} as a function of $\delta>0$, one can deduce that the optimal value for $\delta$ that maximizes $R$, is $\delta^{\ast}=\sqrt{L}$ and gives $R^{\ast}=R(\delta^{\ast})=2\big(\sqrt{L}-\sqrt{L-\mu}\big)-\varepsilon=2\sqrt{\mu}-\varepsilon$ and $\alpha=2\sqrt{L}-\sqrt{L-\mu}-\varepsilon=2\sqrt{\mu}-\varepsilon$, for any $\varepsilon \in (0,\alpha_{-}-\frac{L}{\delta})$, which concludes the first point of Theorem \ref{coroptimal}.

\item 
If $\mu<L$ and  $\delta>0$, then from the conditions of Theorem \ref{basicteoratesconvex} the maximal value that we can choose for $\alpha$, is $\alpha=\alpha_{-}=\frac{1}{2}\bigg(\delta+\frac{3L}{\delta}-\sqrt{\bigg(\delta+\frac{L}{\delta}\bigg)^2-4\mu}\bigg)$.
It follows that the convergence factor $R$ is equal to :
\begin{equation}\label{ratesRfinal2}
R=R(\delta)=2\bigg(\alpha_{-}-\frac{L}{\delta}\bigg) = \delta +\frac{L}{\delta} -\sqrt{\bigg(\delta+\frac{L}{\delta}\bigg)^2-4\mu} \qquad \forall \delta>0
\end{equation}

Finally, in the same way as in the previous case, one can deduce that the optimal value for $\delta$ that maximizes $R$ in \eqref{ratesRfinal2}, is $\delta^{\ast}=\sqrt{L}$ and gives $R^{\ast}=R(\delta^{\ast})=2\big(\sqrt{L}-\sqrt{L-\mu}\big)$ and $\alpha=2\sqrt{L}-\sqrt{L-\mu}$ which concludes the second point of Theorem \ref{coroptimal}
\end{itemize}
\end{proof}

\section*{Acknowledgements}
V.A. acknowledges the financial support of the European Research Council (grant SLING 819789), the AFOSR projects FA9550-
18-1-7009, FA9550-17-1-0390 and BAA-AFRL-AFOSR-2016-0007 (European Office of Aerospace Research and Development), and the EU H2020-MSCA-RISE project NoMADS - DLV-777826. S. V. acknowledges the support of GNAMPA 2020: "Processi evolutivi con memoria descrivibili tramite equazioni integro-differenziali". This work has been supported by the ITN-ETN project TraDE-OPT funded by the European Union’s Horizon 2020 research and innovation programme under the Marie Skłodowska-Curie grant agreement No 861137.

\newpage

\appendix

\section{General lemmas}\label{appendixa}

In this appendix we give some general auxiliary lemmas used in the main core of this work.

\begin{lemma}[Gr\"onwall's Lemma]\label{gronwallcontinuous}
Let $I\subset \R_{+}=[0,+\infty)$ be an interval and $u$,$g$,$h$ $:I\longrightarrow \R$, with $h\in L^{1}_{loc}(I)$, $g\in \mathcal{C}(I)$, and $u\in\mathcal{C}^{1}(I)$, such that for all $t\in I$: \begin{equation}\label{gronwalhypothesis}\dot{u}(t)\leq g(t)u(t)+h(t)
\end{equation}
Then for all $s<t\in I$, it holds: \begin{equation}\label{gronwalconclusion}
u(t)\leq u(s)e^{G(t)}+\int_{s}^{t}e^{G(t)-G(r)}h(r)dr
\end{equation}
where $G(t)=\int_{s}^{t}g(r)dr$.
\end{lemma}

\begin{proof}
By multiplying both sides of relation \eqref{gronwallcontinuous} by $e^{-G(t)}\geq 0$, we obtain :
\begin{equation*}
\big(e^{-G(t)}u(t)\big)^{\prime} \leq e^{-G(t)}h(t)
\end{equation*}
Hence by integrating on $[s,t]$, we find (notice that $G(s)=0$) :
\begin{equation*}
e^{-G(t)}u(t)\leq u(s) +\int_{s}^{t}e^{-G(r)}h(r)dr
\end{equation*}
which by multiplying on both sides by $e^{-G(t)}\geq 0$, gives \eqref{gronwalconclusion} and allows to conclude the proof of the Lemma.
\end{proof}

As a consequence of Lipschitz continuity of the gradient, one can obtain the following inequality (see for example \cite{BertsekasNonlinearProgramming} or \cite{nesterov2013introductory}): 

\begin{lemma}\label{convexLipschitz} Let $F:\Hb\longrightarrow\R$ a function with $L$-Lipschitz gradient. Then 

\begin{equation}\label{descentlemm}
F(x)-F(y)\leq \psc{\nabla F(y)}{x-y} +\frac{L}{2}\norme{x-y}^{2}
\end{equation}
If in addition $F$ is convex, then :
\begin{equation}\label{convexLipschitz2}
    \forall x,y\in \Hb \quad F(y)-F(x)\geq \psc{\nabla F(x)}{y-x}+\frac{1}{2L}\norme{\nabla F(y)-\nabla F(x)}^{2}
\end{equation}

\end{lemma}

\phantomsection 
\addcontentsline{toc}{chapter}{Bibliography} 
\bibliographystyle{plain}
\bibliography{reference2}


\end{document}